\documentclass{amsart}
\usepackage{microtype}
\usepackage[english]{babel}
\usepackage{hyperref}
\usepackage[all]{xy}
\usepackage{bm}
\usepackage{xkeyval}
\usepackage{graphicx}
\usepackage{amssymb,amscd,verbatim, amsthm,amsmath,amsgen,xspace}
\usepackage{latexsym}
\usepackage{amsfonts}
\usepackage{color}
\usepackage{ulem}
\usepackage{exscale}
\usepackage{comment}

\usepackage{ytableau}

\usepackage{amsbsy}

\newtheorem{prop}[equation]{Proposition}
\newtheorem*{prop*}{Proposition}
\newtheorem{thm}[equation]{Theorem}
\newtheorem*{thm*}{Theorem}
\newtheorem{lem}[equation]{Lemma}
\newtheorem*{lem*}{Lemma}

\newtheorem*{kor*}{Corollary}
\newtheorem{cor}[equation]{Corollary}
\newtheorem{rem}[equation]{Remark}

\newtheorem{ex}[equation]{Example}

\numberwithin{equation}{section}

\newcommand{\fre}{\mathfrak{e}}

\newcommand{\frg}{\mathfrak{g}}

\newcommand{\frk}{\mathfrak{k}}

\newcommand{\frp}{\mathfrak{p}}

\newcommand{\fru}{\mathfrak{u}}

\newcommand{\frgl}{\mathfrak{gl}}
\newcommand{\frsl}{\mathfrak{sl}}
\newcommand{\frsu}{\mathfrak{su}}
\newcommand{\frso}{\mathfrak{so}}
\newcommand{\frsp}{\mathfrak{sp}}

\newcommand{\YD}{\operatorname{YD}}

\newcommand{\bbar}{\,|\,}

\def\bbC{\mathbb{C}}

\def\bbR{\mathbb{R}}
\def\bbZ{\mathbb{Z}}

\let\ccdot\cdot
\def\cdot{\hbox to 2.5pt{\hss$\ccdot$\hss}}

\newcommand{\eq}{\begin{equation}}
	\newcommand{\eeq}{\end{equation}}
\newcommand{\eqn}{\begin{equation*}}
	\newcommand{\bmul}{\begin{multline*}}
		\newcommand{\eemul}{\end{multline*}}
	\newcommand{\eeqn}{\end{equation*}}
\newcommand{\pf}{\begin{proof}}
	\newcommand{\epf}{\end{proof}}

\newcommand{\la}{\lambda}

\renewcommand{\phi}{\varphi}

\newcommand{\La}{\Lambda}

\newcommand{\eps}{\varepsilon}
\newcommand{\half}{\frac{1}{2}}
\newcommand{\dom}{\operatorname{dom}}

\let\ssize\scriptstyle
\newif\ifFIRST\newdimen\MAXright\MAXright0pt
\def\sdynkin{\bgroup\eightpoint\dynkin}
\def\endsdynkin{\enddynkin\egroup}
\def\dynkin{\bgroup\FIRSTtrue\hskip.5em\setbox1\hbox{$\diagup$}%
	\setbox2\hbox{$\diagdown$}%
	\setbox0\hbox to2\wd1{\hrulefill}%
	\setbox3\hbox{$\bullet$}%
	\setbox4\hbox{$\times$}%
	\setbox7\hbox{$\circ$}
	\def\whiteroot##1{\ifFIRST\setbox5\hbox{$##1$}\ifdim\wd5>1.3em
		\hskip-.5em\hskip.5\wd5\fi\fi\FIRSTfalse
		\hskip-.25em\raise1.5\wd3\hbox to0pt{\hss\hskip.45em$
			\ssize##1$\hss}\copy7\hskip-.25em\setbox6\hbox{$##1$}
		\MAXright\wd6}
	\def\root##1{\ifFIRST\setbox5\hbox{$##1$}\ifdim\wd5>1.3em%
		\hskip-.5em\hskip.5\wd5\fi\fi\FIRSTfalse%
		\hskip-.25em\raise1.5\wd3\hbox to0pt{\hss\hskip.45em$%
			\ssize##1$\hss}\copy3\hskip-.25em\setbox6\hbox{$##1$}%
		\MAXright\wd6}%
	\def\whitedroot##1{\ifFIRST\setbox5\hbox{$##1$}\ifdim\wd5>1.3em
		\hskip-.5em\hskip.5\wd5\fi\fi\FIRSTfalse
		\hskip-.25em\lower1.8\wd3\hbox to0pt{\hss\hskip.45em$
			\ssize##1$\hss}\copy7\hskip-.25em\setbox6\hbox{$##1$}
		\MAXright\wd6}%
	\def\whiterroot##1{\hskip-.25em\copy7\hbox to0pt{\hskip.3em$\ssize##1$\hss}%
		\hskip-.25em\setbox6\hbox{\hskip.6em$##1##1$}%
		\MAXright\wd6}%
	\def\droot##1{\ifFIRST\setbox5\hbox{$##1$}\ifdim\wd5>1.3em%
		\hskip-.5em\hskip.5\wd5\fi\fi\FIRSTfalse%
		\hskip-.25em\lower1.8\wd3\hbox to0pt{\hss\hskip.45em$%
			\ssize##1$\hss}\copy3\hskip-.25em\setbox6\hbox{$##1$}%
		\MAXright\wd6}%
	\def\rroot##1{\hskip-.25em\copy3\hbox to0pt{\hskip.3em$\ssize##1$\hss}%
		\hskip-.25em\setbox6\hbox{\hskip.6em$##1##1$}%
		\MAXright\wd6}%
	\def\norroot##1{\hskip-.36em\copy4\hbox to0pt{\hskip.3em$\ssize##1$\hss}%
		\hskip-.48em\setbox6\hbox{\hskip.6em$##1##1$}%
		\MAXright\wd6}%
	\def\noroot##1{\ifFIRST\setbox5\hbox{$##1$}\ifdim\wd5>1.3em%
		\hskip-.5em\hskip.5\wd5\fi\fi\FIRSTfalse%
		\hskip-.36em\raise1.5\wd3\hbox to0pt{\hss\hskip.6em$%
			\ssize##1$\hss}\copy4\hskip-.38em\setbox6\hbox{$##1$}%
		\MAXright\wd6}%
	\def\nodroot##1{\ifFIRST\setbox5\hbox{$##1$}\ifdim\wd5>1.3em%
		\hskip-.5em\hskip.5\wd5\fi\fi\FIRSTfalse%
		\hskip-.36em\lower1.8\wd3\hbox to0pt{\hss\hskip.6em$%
			\ssize##1$\hss}\copy4\hskip-.38em\setbox6\hbox{$##1$}%
		\MAXright\wd6}%
	\def\nolink{\hskip\wd0}
	\def\link{\raise.22em\copy0}%
	\def\llink##1{\raise.32em\copy0\hskip-\wd0%
		\raise.12em\copy0\hskip-.5\wd0\hbox to0pt{\hss$##1$\hss}\hskip.5\wd0}%
	\def\lllink##1{\raise.22em\copy0\hskip-\wd0\raise.32em\copy0\hskip-\wd0%
		\raise.12em\copy0\hskip-.5\wd0\hbox to0pt{\hss$##1$\hss}\hskip.5\wd0}%
	\def\rootupright##1{\hbox to0pt{\raise.45em\copy1\hskip-.25em\raise1.3\ht1%
			\hbox{\copy3\hskip.3em$\ssize##1$}\hss}%
		\setbox6\hbox{\hskip.6em\copy1\copy1$##1##1$}%
		\ifdim\MAXright<\wd6\MAXright\wd6\fi}%
	\def\whiterootupright##1{\hbox to0pt{\raise.45em\copy1\hskip-.25em\raise1.3\ht1
			\hbox{\copy7\hskip.3em$\ssize##1$}\hss}
		\setbox6\hbox{\hskip.6em\copy1\copy1$##1##1$}
		\ifdim\MAXright<\wd6\MAXright\wd6\fi}
	\def\norootupright##1{\hbox to0pt{\raise.45em\copy1\hskip-.36em\raise1.3\ht1%
			\hbox{\copy4\hskip.3em$\ssize##1$}\hss}%
		\setbox6\hbox{\hskip.6em\copy1\copy1$##1##1$}%
		\ifdim\MAXright<\wd6\MAXright\wd6\fi}%
	\def\rootdownright##1{\hbox to0pt{\raise-.5em\copy2\hskip-.25em\raise-1.35\ht1%
			\hbox{\copy3\hskip.3em$\ssize##1$}\hss}\setbox6%
		\hbox{\hskip.6em\copy2\copy2$##1##1$}%
		\ifdim\MAXright<\wd6\MAXright\wd6\fi}%
	\def\whiterootdownright##1{\hbox to0pt{\raise-.5em\copy2\hskip-.25em\raise-1.35\ht1
			\hbox{\copy7\hskip.3em$\ssize##1$}\hss}\setbox6
		\hbox{\hskip.6em\copy2\copy2$##1##1$}
		\ifdim\MAXright<\wd6\MAXright\wd6\fi}
	\def\rootdown##1{\hbox to0pt{\hskip-.05em\vrule height.25em depth.65em%
			\hskip-.25em\raise-.95em\hbox{\copy3\hskip.3em$\ssize##1$}\hss}%
		\setbox6\hbox{$##1$}%
		\ifdim\MAXright<\wd6\MAXright\wd6\fi}%
	\def\whiterootdown##1{\hbox to0pt{\hskip-.05em\vrule height.25em depth.65em
			\hskip-.25em\raise-.95em\hbox{\copy7\hskip.3em$\ssize##1$}\hss}
		\setbox6\hbox{$##1$}
		\ifdim\MAXright<\wd6\MAXright\wd6\fi}
	\def\dots{\hskip.5em\cdots\hskip.5em}}%
\def\enddynkin{\ifdim\MAXright>1em\hskip.5\MAXright\else\hskip.5em\fi\egroup}%

\begin{document} 
	
	\title[Fixed infinitesimal character]{Unitary highest weight modules for $\mathfrak{su}(p, q)$ and $\mathfrak{so}^{*}(2n)$ with fixed integral infinitesimal character}
	\author{Pavle Pand\v zi\'c}
	\address[Pand\v zi\'c]{Department of Mathematics, Faculty of Science, University of Zagreb, Bijeni\v cka 30, 10000 Zagreb, Croatia}
	\email{pandzic@math.hr}
	\author{Ana Prli\'c}
	\address[Prli\'c]{Department of Mathematics, Faculty of Science, University of Zagreb, Bijeni\v cka 30, 10000 Zagreb, Croatia}
	\email{anaprlic@math.hr}
	\author{Vladim\'{\i}r Sou\v cek}
	\address[Sou\v cek]{Matematick\'y \'ustav UK, Sokolovsk\'a 83, 186 75 Praha 8, Czech Republic}
	\email{soucek@karlin.mff.cuni.cz}
	\author{V\'it Tu\v cek}
	\address[Tu\v cek]{Department of Mathematics, Faculty of Science, University of Zagreb, Bijeni\v cka 30, 10000 Zagreb, Croatia}
	\email{vit.tucek@gmail.com}
	\date{}
	\thanks{P.~Pand\v zi\'c and A.~Prli\'c were supported by the QuantiXLie Center of Excellence, a project co-financed by the Croatian Government and European Union through the European Regional Development Fund - the Competitiveness and Cohesion Operational Programme (grant PK.1.1.02.0004), by the Croatian Science Foundation (HRZZ), grant no. IP-2025-02-6514, and
		by the European Union – NextGenerationEU through the National Recovery and Resilience Plan 2021-2026. Institutional grant of University of Zagreb Faculty of Science (IK IA 1.1.3. Impact4Math). V.~Tuček was supported by the QuantiXLie Center of Excellence and by the grant GX19-28628X.
		V.~Sou\v cek was supported by the grants GX19-28628X and GA24-10887S of GAČR}
	\begin{abstract} 
		{We classify  unitary highest weight modules with a given integral infinitesimal character for the real Lie algebras $\frsu(p,q)$ and $\frso^*(2n)$. We treat both regular and singular cases. For $\frsu(p,q)$ we identify the unitarizable modules in the Hasse diagrams of the highest weight orbit. Analogous results for the other Hermitian Lie algebras were given in our earlier publications.}
	\end{abstract}

	\maketitle
	
	\section{Introduction}
	
	Studying unitary highest weight modules is a classical topic in representation theory of simple real Lie groups. Typical examples of unitary highest weight modules include holomorphic discrete series representations, but also interesting small representations like the Weil representation of the metaplectic group. These modules appear in many applications, and they are also good examples of unitary modules because they are very concrete, unlike some other families of modules which are rather abstract.
	
	Unitary highest weight modules were first studied by Harish-Chandra \cite{HC1,HC2}, who constructed the holomorphic discrete series and proved that nontrivial unitary highest weight modules exist precisely when $G$ is of Hermitian type, i.e., the symmetric space $G/K$  admits a Hermitian structure (here $K$ is a maximal compact subgroup of $G$). 
	
	In the 1970s several authors were working towards classification of unitary highest weight modules; see for example \cite{KV}, \cite{RV},  \cite{W} and \cite{P2}. The full classification was obtained in 1983, independently by Enright-Howe-Wallach \cite{EHW} and Jakobsen \cite{J}. This did not close the topic and many authors continued to investigate various features of unitary highest weight modules, including \cite{A}, \cite{BGG}, \cite{DES}, \cite{E}, \cite{EJ}, \cite{ES}, \cite{EW}, \cite{HPP}, \cite{HPZ}, \cite{NOT} and \cite{Sa}. 
	
	Recently we revisited the issue of classification, in part joint with G.Savin,  in \cite{PPST1,PPST2,PPST3,PPSST}.
	Our approach differs from the previous ones in that we make more use of the Dirac inequality, and less use of the $K$-type structure of highest weight modules. The Dirac inequality of Parthasarathy, \cite{P1,P2}, is a basic tool in all of the classification schemes, but \cite{EHW} and \cite{J} use it in a minimal fashion, while we use it to the full extent. 
	Some details about this inequality and its refinements can be found in \cite{HP1,HP2}.
	
	Other basic tools used to some extent in all approaches to the classification include Howe's theory of dual pair correspondences \cite{H1,H2,H3} as well as Schmid's description of $K$-types
	in $S(\frp^-)$  \cite{Sch}, where $\frg=\frk\oplus\frp^+\oplus\frp^-$ is the decomposition of the complexified Lie algebra of $G$ into eigenspaces for a nontrivial element of the (one-dimensional) center of the complexified Lie algebra $\frk$ of $K$. 
	
	Our main motivation for trying to get a better understanding of the classification of unitary highest weight modules is the fact that the classifications of \cite{EHW} and \cite{J} are expressed in terms of certain lines in the dual Cartan algebra, with infinitesimal character varying along each line. It is thus not immediately clear how to obtain a more usual classification scheme where one 
	first fixes an infinitesimal character, and then tries to classify  modules with that  infinitesimal character.  We addressed and resolved this issue when the Lie algebra $\frg_0$ of $G$ is $\frsp(2n,\bbR)$ in \cite{PPST3}, and when $\frg_0$ is $\frso(2,n)$, $\fre_{6(-14)}$, or $\fre_{7(-25)}$ in \cite{PPSST}.
	
	In this paper we study the remaining cases, $\frg_0=\frsu(p,q)$ (in Section 2) and $\frg_0=\frso^*(2n)$ (in Section 3). We restrict our attention to the case of integral infinitesimal character, because the most interesting modules in these cases have integral infinitesimal characters (these are the discrete points in the classification, i.e., those modules that are not full generalized Verma modules). 
	
	Throughout the paper we work with the concept of the parameter of a highest weight module $L(\la)$ with highest weight $\la$. This is a special conjugate of the infinitesimal character, which is equal to $\La=\la+\rho$, where $\rho$ is the half sum of positive roots for the root system which is used to define the concept of highest weight modules. (Here, as usual, infinitesimal characters are expressed in terms of elements of the dual Cartan algebra modulo the Weyl group, using results of Harish-Chandra.) By \emph{string} we mean a sequence of numbers which decrease by one, i.e. $(x, x-1, \ldots, x-y)$.
	
	Highest weights and parameters are clearly in bijection, so we can identify the module $L(\la)$ with its parameter $\La$, and talk about ``unitary parameters". Each of our main results, Theorem \ref{unitarity-regular-su(p, q)}, Theorem \ref{unit reg}, Theorem \ref{su(p,q)-singular}, 
	Theorem \ref{so*-zeroin}, Theorem \ref{so*-zeroout}, Theorem \ref{so*-halfinteger-onehalfoout} and Theorem \ref{so*-halfinteger-onehalfin}, gives a list of unitary parameters $\La$ conjugate to a fixed dominant representative $\La^{\dom}$ of the infinitesimal character.

	\section{The case of $\frsu(p,q),\ p \leq q$}
	\label{sec:su(p,q)}
	
	In this case, the complexification of the Lie algebra $\frg_0 = \frsu(p, q)$ is $\frg = \frsl(p + q, \bbC)$ and $\frk = s(\frgl(p, \bbC) \oplus \frgl(q, \bbC))$. We use standard coordinates and the standard positive root system. The positive compact roots are
	\[
	\eps_i - \eps_j, \quad 1 \leq i < j \leq p \text{ or } p+1 \leq i < j \leq n,
	\]
	where $n = p + q$, and the positive noncompact roots are
	\[
	\eps_i - \eps_j, \quad 1 \leq i \leq p \text{ and } p+1 \leq j \leq n.
	\]
	Therefore, the half sum of positive roots is 
	\[
	\rho = \left ( \frac{n-1}{2}, \frac{n-3}{2},  \ldots,  - \frac{n-1}{2} \right). 
	\]
	By definition, the parameter of the irreducible highest weight 
	$(\mathfrak{g}, K)$-module $L(\lambda)$ is
	\[
	\Lambda = \lambda + \rho.
	\]
	
	If $\lambda$ is $\mathfrak{g}$-dominant, then any irreducible highest weight 
	$(\mathfrak{g}, K)$-module with the same infinitesimal character as $L(\lambda)$ 
	is of the form 
	\[
	L\bigl(w(\lambda + \rho) - \rho\bigr)
	\]
	for some 
	\[
	w \in W^{1} = \{\, w \in W_{\mathfrak{g}} \mid w\rho \ \text{is } \mathfrak{k}\text{-dominant} \,\}.
	\]
	The corresponding parameter is $w(\lambda + \rho)$.
	
	As in \cite{PPST3} and \cite{PPSST}, we will denote by $\La^{\dom}$ the $\frg$--dominant representative of the infinitesimal character $\La$, i. e., we have
	\[
	\La^{\dom}_1 \geq \La^{\dom}_2 \geq \ldots \geq \La^{\dom}_n. 
	\]
	All parameters $\La$ corresponding to highest weight $(\frg, K)$-modules must be dominant regular for $\frk$. Therefore, we have
	\[
	\La_1 > \La_2 > \ldots > \La_p; \quad \La_{p + 1} > \La_{p + 2} > \ldots > \La_{n}.
	\]
	The integrality of the parameter $\Lambda^{\mathrm{dom}}$ means that $\La_i - \La_j \in \mathbb{Z} \quad 1 \leq i < j \leq n$.
	Theorems \cite[Theorem 4.6. and Theorem 4.9.]{PPST3} imply that $L(\la)$ is unitary if and only if $\la$ is of the form
	\[
	\la=(\underbrace{\la_1,\dots,\la_1}_{p'},\la_{p'+1},\dots,\la_p\bbar\la_{p+1},\dots,\la_{n-q'},\underbrace{\la_n,\dots,\la_n}_{q'}),
	\]
	where $p'$ and $q'$ are integers such that $1 \leq p' \leq p, 1 \leq q' \leq q$ and $\lambda_1-\la_n \leq  -n+p'+q'-\min(p',q')+1=-n+\max(p',q')+1$ or $\lambda_1-\la_n = -n+p'+q'-i$ for some integer $i\in[0,\min(p',q')-1]$. Since we assumed $\la_1 - \la_n \in \bbZ$, $L(\la)$ is unitary if and only if $\la_1 - \la_n \leq -n + p' + q'$.
	In terms of $\La = \la + \rho$, $L(\la)$ is unitary if and only if $\La$ is of the form
	\begin{align*}
		\La = (&\underbrace{\La_1, \La_1 - 1, \dots,\La_1 - (p' - 1)}_{p'},\, \La_{p'+1}, \dots, \La_p \bbar \\
		& \La_{p+1}, \dots, \La_{n - q'},\, \underbrace{\La_n + (q' - 1), \dots, \La_n + 1, \La_n}_{q'}).
	\end{align*}
	where $\La_{p'+1} < \La_1 - p', \La_{n - q'} > \La_n + q'$ and 
	\[\La_1-\La_n \leq  p' + q' - 1.\] The last condition can be written equivalently as
	\eq\label{unit cond}
	\La_{p'} - \La_{n - q' + 1} \leq 1.
	\eeq
	This description of unitarity will be used throughout this section. We will call the first $p'$ coordinates of $\La$ $p'$--string and the remaining coordinates $(\La_{p'+1}, ..., \La_p)$ will be called $(p-p')$--part and analogously for the $q'$--string and $(q-q')$--part.
	
	\subsection{Regular case}
	
	First, we consider the case when $\La^{\dom}$ is regular. This means that no coordinates of $\La$ repeat.

	\begin{thm} \label{unitarity-regular-su(p, q)}
		Suppose $\Lambda^{\dom}$ is regular integral. Then:
		
		\begin{itemize}
			\item [(a)]   If $\La^{\dom}_{q} = \La^{\dom}_{q+1} + 1$, then there exists $1\leq s \leq q$, $q + 1 - s \leq r \leq n-s$ such that $\La^{\dom}$ is of the form \eqref{eq:dom-string}. The unitary conjugates of $\Lambda^{\dom}$ are 
			\[
			\La = (\La_{q+1}^{\dom}, \La_{q+2}^{\dom}, \ldots, \La_{n}^{\dom} \, | \, \La_{1}^{\dom}, \ldots ,\La_{q}^{\dom})
			\]
			and  
			\begin{align*}
				\Lambda = (&\La^{\dom}_{s} - a, \ldots, \La^{\dom}_{s} - b; \La^{\dom}_{s} - c - 1, \ldots, \La^{\dom}_{s} - r;
				\Lambda^{\mathrm{dom}}_{s+r+1}, \ldots, \Lambda_{n}^{\mathrm{dom}} \, | \nonumber\\
				&\Lambda^{\mathrm{dom}}_{1}, \ldots ,\Lambda^{\mathrm{dom}}_{s-1}; 
				\La^{\dom}_{s}, \ldots, \La^{\dom}_{s} - a + 1; \La^{\dom}_{s} - b - 1, \ldots, \La^{\dom}_{s} - c), 
			\end{align*}
			where $0 \leq a \leq q - s$, $0 \leq b-a \leq (s + r - 1) - q$ and $c = b - a - s + q + 1$. In this case $\La_{p'} = \La^{\dom}_{s} - b$, $\La_{n - q' + 1} = \La^{\dom}_{s} - b - 1$, $p' = b - a + 1$, and $q' = c-b$. If $a = 0$, then the string $\La^{\dom}_{s}, \ldots, \La^{\dom}_{s} - a + 1$ in \eqref{unitary_string} is empty. If $c = r$ then the string $\La^{\dom}_{s} - c - 1, \ldots, \La^{\dom}_{s} - r$ in \eqref{unitary_string} is empty.
			\item [(b)] If $\La^{\dom}_{q} \neq \La^{\dom}_{q+1} + 1$, then the only unitary conjugate is 
			\[
			\La = (\La_{q+1}^{\dom}, \La_{q+2}^{\dom}, \ldots, \La_{n}^{\dom} \, | \, \La_{1}^{\dom}, \ldots ,\La_{q}^{\dom}).
			\]
		\end{itemize}
	\end{thm}

	\begin{proof}
		If $\La \in W^1 \La^{\dom}$ is unitary, then the unitarity condition \eqref{unit cond} and the fact that $\La$ has no repeated coordinates imply that  
		\[
		\text{ either } \La_{p'} < \La_{n - q' + 1}, \text{ or } \La_{p'} = \La_{n - q' + 1} + 1.
		\]
		If $\La_{p'} < \La_{n - q' + 1}$, then, since $\La$ has no repeated coordinates, we must have $\La_1 < \La_n$. Therefore, we have
		\[
		\La_{p+1} > \ldots > \La_{n - q'} >\La_{n - q' + 1} > \ldots > \La_n > \La_1 > \ldots > \La_{p'} > \La_{p' + 1} > \ldots > \La_p.
		\]
		and we see that $\La_n = \La_{q}^{\dom} \neq \La_{q+1}^{\dom}+1.$
		Since the coordinates of $\La$ are obtained by permuting those of $\La^{\dom}$, we have
		\eq
		\label{tilde La} 
		\La = (\La_{q+1}^{\dom}, \La_{q+2}^{\dom}, \ldots, \La_{n}^{\dom} \, | \, \La_{1}^{\dom}, \ldots ,\La_{q}^{\dom}).
		\eeq

		If $\La_{p'} = \La_{n - q' + 1} + 1$, then the sequence $\La_1, \ldots, \La_{p'}, \La_{n - q' + 1}, \ldots, \La_n$ forms a string (a sequence of integers descending by $1$).
		Since $\La_{p'} > \La_{p'+1}$ and $\La$ has no repeated coordinates, we have $\La_{p'+1} < \La_n$. Similarly, since  $\La_{n - q'} > \La_{n - q' + 1}$ and  $\La$ has no repeated coordinates, we have $\La_{n - q'} > \La_1$. Therefore, we have
		\[
		\La_{p+1} > \ldots > \La_{n - q'} > \La_{1} > \ldots > \La_{p'} > \La_{n - q' + 1} > \ldots > \La_{n} > \La_{p' + 1} > \ldots > \La_p.
		\]
		Since the coordinates of $\La$ are obtained by permuting those of $\La^{\dom}$, we have
		\[
		\La = (\La_{q - q'+1}^{\dom}, \ldots, \La_{q - q'+p'}^{\dom}, \La_{q + p'+1}^{\dom}, \ldots, \La_{n}^{\dom} \, | \, \La_{1}^{\dom}, \ldots ,\La_{q - q'}^{\dom}, \La_{q - q'+ p' + 1}^{\dom}, \ldots, \La_{q + p'}^{\dom}), 
		\]
		where the sequence $\La_{q - q'+1}^{\dom}, \ldots, \La_{q - q'+p'}^{\dom}, \La_{q - q'+ p' + 1}^{\dom}, \ldots, \La_{q + p'}^{\dom}$ forms a string such that $\La_q^{\dom}$ belongs to this string and $\La_q^{\dom}$ does not lie on the right end of the string (namely $\La_q^{\dom} \neq \La_{q + p'}^{\dom}$ since $p'\geq 1$). In particular, we have  $\La_{q}^{\dom} = \La_{q+1}^{\dom}+1$ which proves (b).
		
		Therefore, there exist $s$ and $r$ as in the statement of the theorem (namely $s=q-q'+1$ and $r=p'+q'-1$) such that  $\La^{\dom}$ is of the form
		\begin{equation}\label{eq:dom-string}
			\La^{\dom} = (\La^{\dom}_1, \ldots, \La^{\dom}_{s -1}; \La^{\dom}_{s},  \La^{\dom}_{s} - 1 \ldots, \La^{\dom}_{s} - r; \La^{\dom}_{s + r+ 1}, \ldots, \La^{\dom}_n),
		\end{equation}
		where $\La^{\dom}_{s-1} \geq \La^{\dom}_{s} + 2$ if $s>1$, $\La^{\dom}_{s+r+1} \leq \La^{\dom}_{s} - r - 2$ if $n> s + r$, and 
		$s \leq q \leq s + r - 1$. 
		
		If $\La^{\dom}$ is not of the form \eqref{eq:dom-string}, i.e. if $\La^{\dom}_q \neq \La^{\dom}_{q+1} + 1$, then the only unitary parameter conjugate to $\La^{\dom}$ is \eqref{tilde La}. 
		
		If $\La^{\dom}$ is of the form \eqref{eq:dom-string}, i.e. if $\La^{\dom}_q = \La^{\dom}_{q+1} + 1$, then besides the  unitary parameter \eqref{tilde La}, $\La^{\dom}$ has unitary conjugates satisfying $\La_{p'} = \La_{n - q' + 1} + 1$. These unitary conjugates are given by: 
		\begin{align}\label{unitary_string}
			\Lambda = (&\La^{\dom}_{s} - a, \ldots, \La^{\dom}_{s} - b; \La^{\dom}_{s} - c - 1, \ldots, \La^{\dom}_{s} - r;
			\Lambda^{\mathrm{dom}}_{s+r+1}, \ldots, \Lambda_{n}^{\mathrm{dom}} \, | \\
			&\qquad\quad\Lambda^{\mathrm{dom}}_{1}, \ldots ,\Lambda^{\mathrm{dom}}_{s-1}; 
			\La^{\dom}_{s}, \ldots, \La^{\dom}_{s} - a + 1; \La^{\dom}_{s} - b - 1, \ldots, \La^{\dom}_{s} - c), \notag
		\end{align}
		where $0 \leq a \leq b < c \leq r$ and $c = b - a - s + q + 1$. (The last equation comes from the fact that there are exactly $p$ coordinates to the left of the bar and $q$ coordinates to the right of the bar.) Now $b<c$ and $c\leq r$ imply that 
		\[
		a \leq q - s \quad\text{ and }\quad b - a \leq (s + r - 1) - q.
		\]
		In this case $\La_{p'} = \La^{\dom}_{s} - b$, $\La_{n - q' + 1} = \La^{\dom}_{s} - b - 1$, $p' = b - a + 1$, and $q' = c-b$. If $a = 0$, then the string $\La^{\dom}_{s}, \ldots, \La^{\dom}_{s} - a + 1$ in \eqref{unitary_string} is empty. If $c = r$ then the string $\La^{\dom}_{s} - c - 1, \ldots, \La^{\dom}_{s} - r$ in \eqref{unitary_string} is empty.
	\end{proof}
	
	\begin{ex}{\rm 
			If $\La^{\dom} = \rho = \left ( \frac{n-1}{2}, \frac{n-3}{2},  \ldots,  - \frac{n-1}{2} \right)$, then $\rho$ is of the form \eqref{eq:dom-string}, where $s = 1$, $r= n - 1$ and $\La_{1}^{\dom} = \frac{n-1}{2}$. Therefore, unitary parameters are
			\[
			\tilde{\rho} = \left ( \frac{n-1}{2} - q, \frac{n-1}{2} - q - 1, \ldots, -\frac{n-1}{2} \, \middle| \, \frac{n-1}{2}, \frac{n-1}{2} - 1, \ldots, \frac{n-1}{2} - q + 1 \right ) 
			\]
			and the parameters of the form    
			\begin{align}\label{unitary_string-rho}
				\Lambda = \Bigl( &
				\frac{n-1}{2} - a, \ldots,  \frac{n-1}{2} - b;
				\frac{n-1}{2} - c - 1, \ldots,  -\frac{n-1}{2} 
				\, \Big| \\
				& \qquad\quad\frac{n-1}{2}, \ldots,  \frac{n-1}{2} - a + 1;
				\frac{n-1}{2} - b - 1, \ldots,  \frac{n-1}{2} - c
				\Bigr)\nonumber,
			\end{align}
			where $0 \leq a \leq q - 1$, $0 \leq b-a \leq n - 1 - q = p - 1$ and $c = b - a + q$. In this case $\La_{p'} =  \frac{n-1}{2} - b$, $\La_{n - q' + 1} =  \frac{n-1}{2} - b - 1$, $p' = b - a + 1$, and $q' = c-b$. If $a = 0$, then the string $ \frac{n-1}{2}, \ldots,  \frac{n-1}{2} - a + 1$ in \eqref{unitary_string-rho} is empty. If $c = n-1$ then the string $ \frac{n-1}{2} - c - 1, \ldots,  -\frac{n-1}{2}$ in \eqref{unitary_string-rho} is empty.
			
			In case $a = 0$ and $c = n-1$, i.e. $b = p - 1$ the parameter \eqref{unitary_string-rho} is equal to $\rho$. 
		}
	\end{ex}
	
	To get a better understanding of these unitary points (other than $\tilde\rho$), we give them a geometric interpretation as ``edge points" of the Hasse diagram. 
	For this, it is useful to attach a Young diagram $\YD_w$ to each $w\rho$. For $1\leq k\leq p$, we let the $k$-th row of $\YD_w$ consist of
	\[
	(w\rho)_k-\tilde\rho_k
	\]
	boxes. In this way we clearly get a Young diagram with at most $p$ rows and with each row of length at most $r-(k-1) - (p-r-k)=q$. Note that $\tilde\rho$ corresponds to the empty diagram, while $\rho$ corresponds to the maximal diagram with $pq$ boxes arranged in a $p\times q$ rectangle.
	
	We say that $w\rho$ is the $i$-th point on the $j$-th edge of the Hasse diagram if $\YD_w$ has $j$ nonzero rows, each with $i$ boxes. This means that $w\rho$ is equal to $\La$ in 
	\eqref{unitary_string-rho}, with $i= q' = c - b$ and $j= p' = b - a + 1$.
	
	\begin{ex}{\rm Let $p=2$ and $q=3$. Then $\rho = (2, 1, 0, -1, -2)$. The Hasse diagram is
			
			\medskip
			
			{\footnotesize
				\[
				\begin{CD}
					(2,1\,|\,0,-1,-2)@<<<(2,0\,|\,1,-1,-2)@<<<(2,-1\,|\,1,0,-2)@<<<(2,-2\,|\,1,0,-1) \\
					@. @AAA @AAA @AAA \\
					@.(1,0\,|\,2,-1,-2)@<<<(1,-1\,|\,2,0,-2)@<<<(1,-2\,|\,2,0,-1) \\
					@.@.@AAA @AAA \\
					@.@.(0,-1\,|\,2,1,-2)@<<<(0,-2\,|\,2,1,-1)\\
					@.@.@.@AAA\\
					@.@.@.(-1,-2\,|\,2,1,0)
				\end{CD}
				\]
			}
			\medskip
			
			\noindent with arrows pointing to larger elements in Bruhat order. The corresponding Young diagrams are
			
			\medskip
			
			{\tiny
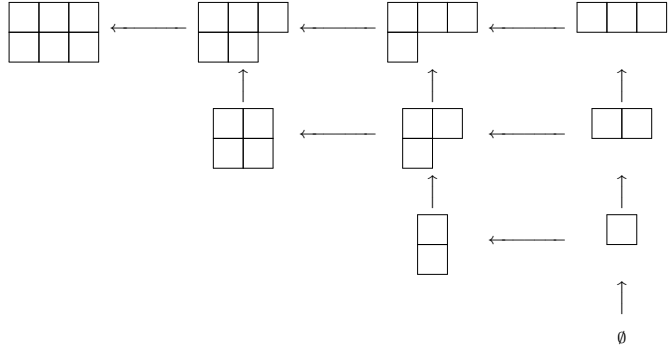
				\[
				\begin{CD}
					\ydiagram{3,3}@<<<\ydiagram{3,2}@<<<\ydiagram{3,1}@<<<\ydiagram{3,0} \\
					@. @AAA @AAA @AAA \\
					@.\ydiagram{2,2}@<<<\ydiagram{2,1}@<<<\ydiagram{2,0} \\
					@.@.@AAA @AAA \\
					@.@.\ydiagram{1,1}@<<<\ydiagram{1,0}\\
					@.@.@.@AAA\\
					@.@.@. \emptyset
				\end{CD}
				\]
			}
			\medskip
			
			\noindent The first edge is the last column, and the second edge is the diagonal, both excluding the smallest point $\tilde\rho$. The points on each edge are counted in the direction of arrows.
		}
	\end{ex}
	
	We now go back to general regular integral parameters. As mentioned earlier, they are all of the form 
	\[
	w(\rho+\mu),\qquad w\in W^1,\ \mu \text{ a $\frg$-dominant weight}.
	\]

	\begin{thm}
		\label{unit reg}
		The parameter $w(\rho+\mu)$ is unitary if and only if the following two conditions are satisfied:
		\begin{enumerate}
			\item $w\rho$ is a unitary point in the Hasse diagram of $\rho$, i.e., either $w\rho=\tilde\rho$, or $w\rho$ is an edge point;
			\item If $w\rho$ is the $i$-th point on the $j$-th edge of the Hasse diagram, then 
			\[
			(w\mu)_1=\dots=(w\mu)_j=(w\mu)_{n-i+1}=\dots=(w\mu)_n,
			\]
			with dots standing for consecutive coordinates. 
		\end{enumerate}
		
		In other words, if $w\rho$ is the $i$-th point on the $j$-th edge of the Hasse diagram of $\rho$, then the regular integral modules with parameters in the Weyl chamber of $w\rho$ are unitary exactly in the reduced translation cone of dimension $p+q-i-j$, described by condition (2).
		In the chamber of $\tilde\rho$, the full translation cone consists of unitary modules. 
	\end{thm}
	
	\begin{proof}
		It follows from Theorem \ref{unitarity-regular-su(p, q)} that if $\La^{\dom}_{q} = \La^{\dom}_{q + 1} + 1$, then the unitary parameters are 
		\eq
		\label{cone 1}
		w (\rho + \mu) = \La = (\La_{q+1}^{\dom}, \La_{q+2}^{\dom}, \ldots, \La_{n}^{\dom} \, | \, \La_{1}^{\dom}, \ldots ,\La_{q}^{\dom})
		\eeq
		and  
		\begin{align}
			\label{cone 2}
			w (\rho + \mu) = \Lambda = (&\La^{\dom}_{s} - a, \ldots, \La^{\dom}_{s} - b; \La^{\dom}_{s} - c - 1, \ldots, \La^{\dom}_{s} - r;
			\Lambda^{\mathrm{dom}}_{s+r+1}, \ldots, \Lambda_{n}^{\mathrm{dom}} \, | \\
			&\Lambda^{\mathrm{dom}}_{1}, \ldots ,\Lambda^{\mathrm{dom}}_{s-1}; 
			\La^{\dom}_{s}, \ldots, \La^{\dom}_{s} - a + 1; \La^{\dom}_{s} - b - 1, \ldots, \La^{\dom}_{s} - c)\nonumber \\
			= (& \La^{\dom}_{s + a}, \ldots, \La^{\dom}_{s + b}; \La^{\dom}_{s + c + 1}, \ldots, \La^{\dom}_{n} \, | \, \La^{\dom}_{1}, \ldots, \La^{\dom}_{s + a - 1}; \La^{\dom}_{s + b + 1}, \ldots, \La^{\dom}_{s + c}), \nonumber
		\end{align}
		where $0 \leq a \leq q - s$, $0 \leq b-a \leq (s + r - 1) - q$ and $c = b - a - s + q + 1$, $p' = b - a + 1$, and $q' = c-b$. 
		
		In case $w (\rho + \mu)$ is given by \eqref{cone 1}, we have
		\[
		w \rho = \tilde{\rho}.
		\]
		If $w (\rho + \mu)$ is given by \eqref{cone 2}, then
		\[
		w \mu = (\mu_{s + a}, \ldots, \mu_{s + b}; \mu_{s + c + 1}, \ldots, \mu_n \, | \, \mu_1, \ldots, \mu_{s + a - 1}; \mu_{s + b + 1}, \ldots, \mu_{s + c}).
		\]
		Therefore, we have 
		\begin{align*}
			\Lambda & = w \rho + w \mu  \\
			& = \Bigl( 
			\frac{n-1}{2} - (s + a) + 1, \ldots,  \frac{n-1}{2} - (s + b) + 1;
			\frac{n-1}{2} - (s + c), \ldots,  -\frac{n-1}{2} 
			\, \Big| \nonumber\\
			& \qquad \frac{n-1}{2}, \ldots,  \frac{n-1}{2} - (s + a) + 2;
			\frac{n-1}{2} - (s + b), \ldots,  \frac{n-1}{2} - (s + c) + 1
			\Bigr) \\
			& \qquad  + (\mu_{s + a}, \ldots, \mu_{s + b}; \mu_{s + c + 1}, \ldots, \mu_n \, | \, \mu_1, \ldots, \mu_{s + a - 1}; \mu_{s + b + 1}, \ldots, \mu_{s + c}). 
		\end{align*}
		Comparing with \eqref{cone 2}, it follows that 
		\begin{align*}
			\mu_{s + a} + \frac{n-1}{2} - s -  a + 1 & = \La^{\dom}_{s} - a \\    
			& \vdots \\
			\mu_{s + b} + \frac{n-1}{2} - s -  b + 1 & = \La^{\dom}_{s} - b \\
			\mu_{s + b + 1} + \frac{n-1}{2} - s - b & = \La^{\dom}_{s} - b - 1 \\
			& \vdots \\
			\mu_{s + c} + \frac{n-1}{2} - s - c  + 1& = \La^{\dom}_{s} - c.
		\end{align*}
		Therefore we have
		\[
		\mu_{s + a} = \ldots = \mu_{s + b} = \mu_{s + b + 1} = \ldots = \mu_{s + c} = \Lambda^{\dom}_s - \frac{n-1}{2} + s - 1.
		\]
		Since $i = q' = c - b$ and $j = p' = b - a + 1$, 
		we have
		\[
		(w\mu)_1=\dots=(w\mu)_j=(w\mu)_{n-i+1}=\dots=(w\mu)_n.
		\]
	\end{proof}
	
	\subsection{Singular case}
	
	The singular parameters are those $\La^{\dom}$ which have at least one repeated coordinate. If a coordinate is repeated more than twice then there is no paramater in its orbit that would be dominant regular for $\frk.$ 
	
	First we will prove that in order to have any unitary conjugate, $\Lambda^{\mathrm{dom}}$ must be
	of the form
	\begin{equation} \label{lad_su(p,q)-singular}
		\La^{\dom} = (D_1;\, C_1\, A_1^{'} B_1 A_2^{'} B_2 \cdots B_{r-1} A_r^{'}\, C_2;\, D_2),\quad r \ge 1
	\end{equation}
	where the middle part
	\[
	C_1\, A_1^{'} B_1 A_2^{'} B_2 \cdots B_{r-1} A_r^{'} \, C_2
	\]
	is a weakly decreasing string and it decreases strictly except for the repetitions
	inside the blocks $A_i^{'}$. Each substring $A_i^{'}$ consists of coordinates that
	appear exactly twice, while all coordinates in the substrings $B_i$, $C_1$, and
	$C_2$ are non-repeated. 
	
	Moreover:
	\begin{itemize}
		\item Each $A_i^{'} \neq \emptyset$ is a weakly decreasing string of pairs of repeated coordinates.
		\item For each $A_i^{'}$, we will denote by $A_i$ the underlying strictly decreasing string of distinct coordinates. The number of repeated coordinates is equal to $\sum_{i = 1}^{r} |A_i|$, and this number is at most $p$, since in a unitary conjugate each repeated coordinate must appear on both the left and the right sides of the bar, and $p \leq q$.
		\item Each $B_i \neq \emptyset$ consists of (strictly decreasing) string non-repeated coordinates. (It is possible that $r = 1$; in that case, there are no $B_i$ strings.)
		\item $C_1$ and $C_2$ are, possibly empty, strictly decreasing strings such that there is no gap with $A_1^{'}$ and $A_r^{'}$, respectively.
		\item The sequences $D_1$ and $D_2$ are (possibly empty) decreasing sequences of non-repeated coordinates, not necessarily strings. There is a gap of at least two between $D_1$ and $C_1$ when $D_1,C_1 \neq \emptyset$, or between $D_1$ and $A_1'$ when $D_1 \neq \emptyset$, $C_1 =\emptyset$. Likewise, there is a gap of at least two between $C_2$ and $D_2$ when $C_2,D_2 \neq \emptyset$, or between $A_r'$ and $D_2$ when $D_2 \neq \emptyset$, $C_2 =\emptyset$. 
	\end{itemize}

	\begin{ex} We have
		\[
		(
		\underbrace{18, 16, 15}_{D_1};
		\underbrace{12, 11}_{C_1},
		\underbrace{10, 10, 9, 9, 8, 8}_{A_1^{'}}, 
		\underbrace{7, 6}_{B_1}, 
		\underbrace{5, 5}_{A_2^{'}},
		\underbrace{4}_{B_2}, 
		\underbrace{3, 3, 2, 2}_{A_3^{'}},
		\underbrace{1}_{B_3},
		\underbrace{0, 0, -1, -1}_{A_4^{'}}; 
		\underbrace{-5}_{D_2}
		)
		\]
		with
		\[
		A_1 = (10, 9, 8), \ A_2 = (5), \ A_3 = (3, 2), \ A_4 = (0, -1), \  C_2 = \emptyset.
		\]
	\end{ex}
	
	The claimed form \eqref{lad_su(p,q)-singular} follows from the following three lemmas. Recall that any unitary conjugate $\La$ of $\La^{\dom}$ must satisfy the unitarity condition \eqref{unit cond}.
	
	\begin{lem}\label{sing upq l1}
		If $\La$ is unitary, then:
		
		(1) The $(p-p')$--part and $(q-q')$--part of $\La$ do not share any coordinates. 
		
		(2) Every repeated coordinate must appear in the union of the $p'$--string and the $q'$--string.
	\end{lem}
	\pf (1) follows from the following chain of inequalities
	\[
	\underbrace{\La_p < \ldots < \La_{p' + 1}}_{(p - p') -\text{part }} < \La_{p'} \leq \La_{n - q' + 1} + 1 < \underbrace{\La_{n - q'} < \La_{n - q' - 1} < \ldots < \La_{p+1}}_{(q - q') -\text{part }},
	\]
	where the $\leq$ inequality is the unitarity condition \eqref{unit cond} and the other inequalities are obvious.
	
	(2) follows immediately from (1).
	\epf
	
	Let us denote by $I$ the set of shared coordinates of the $p'$--string and the $(q-q')$--part. 
	
	If the $p'$--string and the $q'$--string share some coordinates, then they share the entire string between the highest and the lowest coordinates that appear in both the $p'$--string and the $q'$--string. Let us denote that string by $II$.
	
	Let us denote by $III$ the set of shared coordinates of the $(p-p')$--part and the $q'$--string. Some of the sets $I, II, III$ may be empty, but $I \cup II \cup III \neq \emptyset$, since we are considering the singular case. 
	If all three of them are not empty, we have 
	\[I > II > III,\]
	meaning that all elements of $I$ are greater than all elements of $II$ and all elements of $II$ are greater than all elements of $III$.
	
	Now we will determine what form do the unitary conjugates of $\La^{\dom}$ take. The answer is that in the case of $II \neq \emptyset$ the only freedom can be in distribution of elements of $C_1$ and $C_2$ to the left and right of the bar. In the case of $II = \emptyset$, there can be additional freedom in distribution of elements of $B_i$ where $i$ is maximal such that $A_1B_1\ldots A_i B_i$ forms a string.
	
	\begin{lem}
		\label{sing upq l2}
		Assume that $\La^{\dom}$ is singular and that $\La$ is a unitary conjugate of $\La^{\dom}$. Then:
		
		(1) Either $\La_{p'} -  \La_{n - q' + 1} \leq 0$ or $\La_{p'} -  \La_{n - q' + 1} = 1$.
		
		(2) If $\La_{p'} -  \La_{n - q' + 1} \leq 0$, then the set $II$ is not empty. 
		
		(3) If $\La_{p'} -  \La_{n - q' + 1} = 1$, then the sequence $\La_1, \ldots, \La_{p'}, \La_{n - q' + 1}, \ldots, \La_{n}$ forms a string, and $II = \emptyset$.
	\end{lem}
	\pf (1) is immediate from \ref{unit cond}.
	
	(2) Suppose that $II$ is empty.
	Then $\La_{p'}$ is not contained in the $q'$--string, and since by assumption $\La_{p'} \leq  \La_{n - q' + 1}$, it follows that $\La_{p'} < \La_{n}$. 
	
	Likewise, $\La_n$ is not contained in the $p'$--string, and thus $\La_n > \La_1$. Consequently, we have
	\[
	(q - q')-\text{part } > q'-\text{string } > p' -\text{string } > (p - p') -\text{part},
	\]
	which is a contradiction with $\La^{\dom}$ being singular. 
	
	(3) This statement is obvious.
	\epf
	
	\begin{lem}
		\label{sing upq l3}
		Suppose that $\La^{\dom}$ is singular and that $\La$ is a unitary conjugate of $\La^{\dom}$. Then:
		
		(1) By arranging the coordinates of $\La$ in the union of the $p'$--string and the $q'$--string in decreasing order, we obtain one large string in which some coordinates appear twice and some only once. 
		
		(2) The part of $\La^{\dom}$  between the highest and lowest repeated coordinates must consist of a string in which some coordinates appear twice and some only once.
	\end{lem}
	\pf (1) By Lemma \ref{sing upq l2}, there are two cases. Either $II$ is not empty, meaning that a substring of the $p'$--string will overlap with the $q'$--string, so the union forms a string with overlapping coordinates repeating twice, or $II=\emptyset$ and the union of the $p'$--string and the $q'$--string consists of one large string with no repetitions.
	
	(2) follows from (1) and the fact that every repeated coordinate must appear in the union of the $p'$--string and the $q'$--string (Lemma \ref{sing upq l1}). 
	\epf
	
	This concludes the proof of \eqref{lad_su(p,q)-singular}. Now we will identify which parts of $\La$ are contained where in the decomposition of $\La^{\dom}$ given by \eqref{lad_su(p,q)-singular}.
	
	\begin{lem}
		\label{sing upq l4}
		Suppose that $\La^{\dom}$ is singular and that $\La$ is a unitary conjugate of $\La^{\dom}$. Then both the
		$p'$--string and the $q'$--string of $\La$ are contained in the string
		\eq\label{sing upq eq1}
		(C_1 A_1 B_1 A_2 B_2 \ldots B_{r-1} A_r C_2)
		\eeq
		(the notation as above).
	\end{lem}
	\pf
	Suppose that the $p'$--string is not contained in the string \eqref{sing upq eq1}.
	Then either $\Lambda_1 \in D_1$ or $\Lambda_{p'} \in D_2$. (See \eqref{lad_su(p,q)-singular}.) Recall that there is a gap
	between $D_1$ and $C_1$ (or between $D_1$ and $A_1$ if $C_1=\emptyset$), and between $C_2$ and $D_2$ (or between $A_r$ and $D_2$ if $C_2=\emptyset$). Therefore, if $\Lambda_1 \in D_1$,
	then the $p'$--string is contained inside $D_1$, and if $\Lambda_{p'} \in D_2$,
	then the $p'$--string is contained inside $D_2$.
	
	If the $p'$--string is contained in $D_2$, then the substrings $A_1, \ldots, A_r$
	do not lie on the left-hand side of the bar in $\La$, being bigger than $\La_1$. This is a contradiction, since repeated coordinates must appear on both sides of the bar. Hence the $p'$--string cannot be
	contained in $D_2$.
	
	If the $p'$--string is contained in $D_1$, then the substrings $A_1, \ldots, A_r$
	lie in the $(p - p')$--part, and therefore, since $\Lambda$ is unitary, the
	substrings $A_1, \ldots, A_r$ must lie inside the $q'$--string (because the
	$(p - p')$--part and the $(q - q')$--part do not share any coordinates, see Lemma \ref{sing upq l1}). Hence, the
	$q'$--string must be contained in the string \eqref{sing upq eq1}. But then
	$\Lambda_{p'} - \Lambda_{n - q' + 1} \ge 2$, which contradicts the assumption
	that $\Lambda$ is unitary, by the unitarity condition \eqref{unit cond}.
	
	Therefore, the $p'$--string is contained in the string
	$(C_1 A_1 B_1 A_2 B_2 \ldots B_{r-1} A_r C_2)$. Using a similar argument, we
	conclude that the $q'$--string is also contained inside the string
	$(C_1 A_1 B_1 A_2 B_2 \ldots B_{r-1} A_r C_2)$. 
	\epf
	
	\begin{cor}
		\label{sing upq cor2}
		Suppose that $\La^{\dom}$ is singular and that $\La$ is a unitary conjugate of $\La^{\dom}$. Let $D_1,D_2$ be as in \eqref{lad_su(p,q)-singular}.
		Then $D_1$ is contained in the $(q - q')$--part of $\La$ and $D_2$ is contained in the $(p - p')$--part of $\La$.
	\end{cor}
	\pf This follows immediately from Lemma \ref{sing upq l4}, since this lemma implies that the coordinates in $D_1$ are bigger than both the $p'$--string and the $q'$--string of $\La$, while the coordinates in $D_2$ are smaller than both the $p'$--string and the $q'$--string of $\La$.
	\epf
	
	Now we can restrict the form of parameters for singular unitary modules.
	
	\begin{thm} \label{su(p,q)-singular}
		If $\Lambda^{\mathrm{dom}}$ is not of the form \eqref{lad_su(p,q)-singular}, 
		then $\Lambda^{\mathrm{dom}}$ does not have any unitary conjugates.  
		If $\Lambda^{\mathrm{dom}}$ is of the form \eqref{lad_su(p,q)-singular}, 
		then the unitary conjugates of $\Lambda^{\mathrm{dom}}$ are exactly those among 
		the candidates  \eqref{su(p,q)-1}--\eqref{su(p,q)-7} 
		that have exactly $p$ coordinates on the left-hand side of the bar.
		\qed\end{thm}
	\pf 
	Given the last corollary, we know that $D_2$ must form the last coordinates left of the bar and $D_1$ must form the first coordinates right of the bar, i.e. $\La = (\ldots D_2|D_1\ldots).$ As we will see in the next paragraph, placement of $A_i$ coordinates is quite restricted so the rest of the proof is occupied by discussing the remaining freedom in the placement of coordinates from $C_1$ and $C_2$.
	
	If, for some $i$, the string $A_i$ and the $p'$--string share at least one coordinate, then all coordinates of $A_i$ must lie in the $p'$--string, since $A_i$ has to appear on the left-hand side of the bar and there is a gap between the $p'$--string and the $(p-p')$--part. Using a similar argument, we conclude that if, for some $i$, the string $A_i$ and the $q'$--string share at least one coordinate, then all coordinates of $A_i$ must lie in the $q'$--string.
	
	If for some $i, j \in \{1,\ldots,r\}$ with $i<j$, both substrings $A_i$ and $A_j$ lie inside the $p'$--string (resp.\ the $q'$--string), then the $p'$--string (resp.\ the $q'$--string) must also contain all substrings that lie between $A_i$ and $A_j$.
	Also, note that the intersection of the $p'$-- and $q'$--strings, if nonempty, is equal to exactly one of the substrings $A_i$, since the coordinates between different $A_i$ are nonrepeated.

	\textbf{Case of $\mathbf{II \neq \emptyset}$:} We have $II = A_i$ for some $i \in \{1, \ldots, r\}$. 
	Then the substrings $A_1, \ldots, A_i$ are contained in the $p'$--string, 
	the substrings $A_{i+1}, \ldots, A_r$ are contained in the $(p - p')$--part,
	the substrings $A_1, \ldots, A_{i-1}$ are contained in the $(q - q')$--part, 
	and the substrings $A_i, A_{i+1}, \ldots, A_r$ are contained in the $q'$--string. 
	
	Therefore, the string 
	\[
	A_1 B_1 \ldots B_{i-1} A_i
	\]
	is contained in the $p'$--string, and the string
	\[
	A_i B_i \ldots B_{r-1} A_r
	\]
	is contained in the $q'$--string. Hence, if $\La_{p'} \geq \La_n$, then  $\Lambda$ must be of the form
	\begin{align}\label{su(p,q)-1}
		\Lambda 
		= &( \underbrace{C_{12}\, A_1 B_1 \ldots B_{i-1} A_i}_{p' -\text{string}}; 
		A_{i+1} \ldots A_r C_{22}  D_2 
		\,|\, \\
		&\qquad\qquad\qquad\qquad D_1 C_{11} A_1 \ldots A_{i-1}; \underbrace{A_i B_i \ldots B_{r-1} A_r C_{21}}_{q'- \text{string}}) \notag,
	\end{align}
	or of the form
	\begin{align}\label{su(p,q)-2}
		\Lambda 
		= ( \underbrace{A_1 }_{p' -\text{string}}; 
		A_{2} \ldots A_r C_{22}  D_2 
		\,|\, 
		D_1; \underbrace{C_1 A_1 B_1 \ldots B_{r-1} A_r C_{21}}_{q' -\text{string}}),
	\end{align}
	where $C_1 = (C_{11}, C_{12})$ and $C_2 = (C_{21}, C_{22})$, with some of the $C_{ij}$ possibly empty. If $i = r$ and $C_2$ is not an empty string, then $C_{22} \neq C_2$, since there is a gap between the $p'$--string and the $(p - p')$--part. Similarly, if $i = 1$ and $C_1$ is not an empty string, then in \eqref{su(p,q)-1} $C_{11} \neq C_1$, since there is a gap between the $(q - q')$--part and the $q'$--string.
	
	If $\La_{p'} < \La_n$, then  $\Lambda$ must be of the form
	\eq\label{su(p,q)-3}
	\Lambda 
	= ( \underbrace{C_{12}\, A_1 B_1 \ldots B_{r-1} A_r C_2}_{p' -\text{string}}; 
	D_2 
	\,|\, 
	D_1 C_{11} A_1 \ldots A_{r-1}; \underbrace{A_r}_{q' -\text{string}}). 
	\eeq
	If $r = 1$ and $C_1$ is not an empty string, then $C_{11} \neq C_1$, since there is a gap between the $(q - q')$--part and the $q'$--string. Also, if $r = 1$, we have one more possibility
	\begin{equation}\label{su(p,q)-4}
		\Lambda 
		= (\underbrace{A_1  C_2}_{p' -\text{string}}; 
		D_2 
		\,|\, 
		D_1;  \underbrace{ C_1 A_1}_{q' -\text{string}}). 
	\end{equation}
	
	\textbf{Case of $\mathbf{II = \emptyset}$:} In this case $\Lambda_1, \ldots, \Lambda_{p'}, \Lambda_{n - q' + 1}, \ldots, \Lambda_n$ form a strictly decreasing string. 
	If none of the $A_i$ is contained in the $p'$--string, then the substrings 
	$A_1, A_2, \ldots, A_r$ are contained in the $(p-p')$--part and the $q'$--string. 
	In that case $\Lambda$ is of the form
	\eq\label{su(p,q)-5}
	\Lambda 
	= (\underbrace{C_{11}}_{p'\text{-string}};\ 
	A_1 A_2 \ldots A_r\, C_{22}\, D_2 
	\,|\, 
	D_1;\ \underbrace{C_{12}\, A_1 B_1 \ldots B_{r-1} A_r\, C_{21}}_{q'\text{-string}} ), 
	\eeq
	where $C_{11} \neq C_{1}$, since there is a gap between the $p'$--string and the $(p - p')$--part. Also, we have $C_{11} \neq \emptyset$.  Therefore, if $|C_1| \leq 1$, then the above candidate is not unitary.
	
	Now we consider the case when some of the $A_i$ are contained in the $p'$--string. 
	Let $i \in \{1, \ldots, r\}$ be the maximal index such that the substring $A_i$ is 
	contained in the $p'$--string. Then the substrings $A_1, \ldots, A_i$ are contained 
	in the $p'$--string and the $(q - q')$--part, while the substrings 
	$A_{i+1}, \ldots, A_r$ are contained in the $(p - p')$--part and the $q'$--string. 
	
	If $i \neq r$, then $\Lambda$ is of the form
	\begin{align}\label{su(p,q)-6}
		\Lambda 
		= (& \underbrace{C_{12}\, A_1 B_1 \ldots B_{i-1} A_i\, B_{i1}}_{p'\text{-string}};\ 
		A_{i+1} \ldots A_r\, C_{22}\, D_2 
		\,|\, \\
		& \qquad\qquad\qquad D_1\, C_{11}\, A_1 \ldots A_i;\ 
		\underbrace{B_{i2}\, A_{i+1} \ldots B_{r-1} A_r\, C_{21}}_{q'\text{-string}} ), \notag
	\end{align}
	where $B_i = (B_{i1}, B_{i2})$, and both $B_{i1}$ and $B_{i2}$ are nonempty strings, 
	since there is a gap between the $p'$--string and the $(p - p')$--part, and between 
	the $(q - q')$--part and the $q'$--string. Therefore, if $|B_i| = 1$, then the above candidate is not unitary. 
	
	If $i = r$, then $\La$ is of the form
	\eq\label{su(p,q)-7}
	\Lambda 
	= ( \underbrace{C_{12}\, A_1 B_1 \ldots B_{r-1} A_r C_{21}}_{p' -\text{string}}; 
	D_2 
	\,|\,  D_1 C_{11} A_1 \ldots A_{r}; \underbrace{C_{22}}_{q' -\text{string}}), 
	\eeq
	where $C_{22} \neq C_2$, since there is a gap between the $(q - q')$--part and the $q'$--string.  Therefore, if $|C_2| \leq 1$, then the above candidate is not unitary.
	\epf
	
	
	\begin{ex}
		{\rm
			Let $p = 10, q = 16$ and let
			\[
			\La^{\dom} = (18, 16, 15;12, 11, 10, 10, 9, 9, 8, 8, 7, 6, 5, 5, 4, 3, 3, 2, 2, 1, 0, 0, -1, -1; -5).
			\]
			Then 
			\begin{align*}
				A_1^{'} & = (10, 10, 9, 9, 8, 8), \, B_1 = (7, 6), A_2^{'} = (5, 5), B_2 = (4), \\
				A_3^{'} & = (3, 3, 2, 2), B_3 = (1), A_4^{'} = 
				(0, 0, -1, -1),
			\end{align*}
			and $A_1 = (10, 9, 8), A_2 = (5), A_3 = (3, 2), A_4 = (0, -1)$,
			$C_1 = (12, 11), C_2 = \emptyset, D_1 = (18, 16, 15)$, and $D_2 = (-5)$. 
			Since $|A_1| + |A_2| + |A_3| + |A_4| = 8$ and $p = 10$, and since $D_2$ must appear on 
			the left-hand side of the bar if $\Lambda$ is unitary, it follows that, besides the 
			coordinates $10, 9, 8, 5, 3, 2, 0, -1, -5$, there is room for only one more coordinate 
			on the left-hand side. Therefore, the string $B_1$ cannot be contained in the $p'$--string, 
			and either the $p'$-string contains only $A_1$ while $A_2, A_3, A_4$ lie in the $(p-p')$--part, 
			or all substrings $A_1, A_2, A_3, A_4$ lie in the $(p-p')$--part.
			
			If all the substrings $A_1, A_2, A_3, A_4$ are inside the $(p - p')$--part, 
			then the $p'$--string has only one coordinate. Since $D_1$ is inside the 
			$(q - q')$--part, the $p'$--string can be either $(12)$ or $(11)$. However, there must be a gap between the $p'$--string and the $(p - p')$--part, so the $p'$--string must be $(12)$. So
			\[
			\Lambda = \mathbf{(12;\ 10, 9, 8, 5, 3, 2, 0, -1, -5 \mid 
				18, 16, 15;\ 11, 10, 9, 8, 7, 6, 5, 4, 3, 2, 1, 0, -1)}.
			\]
			This parameter is of type \eqref{su(p,q)-5}.
			
			If the $p'$--string contains $A_1$ while $A_2, A_3, A_4$ lie in the $(p - p')$--part, 
			then, since $B_1$ is not entirely contained in the $p'$--string, the $p'$--string must be 
			$(10, 9, 8, 7)$, or $(11, 10, 9, 8)$, or $(10, 9, 8)$.
			
			If the $p'$--string is $(10, 9, 8, 7)$, then 
			\[
			\Lambda = \mathbf{(10, 9, 8, 7; \ 5, 3, 2, 0, -1, -5 \mid 
				18, 16, 15, 12, 11, 10, 9, 8;\ 6, 5, 4, 3, 2, 1, 0, -1)}.
			\]
			This parameter is of type \eqref{su(p,q)-6}.
			
			If the $p'$--string is $(11, 10, 9, 8)$, then 
			\[
			\Lambda = \mathbf{(11,10, 9, 8; \ 5, 3, 2, 0, -1, -5 \mid 
				18, 16, 15, 12;\ 10, 9, 8, 7, 6, 5, 4, 3, 2, 1, 0, -1)}.
			\]
			
			This parameter is of type \eqref{su(p,q)-1}.
			
			If the $p'$--string is $(10, 9, 8)$, then $18, 16, 15, 12, 11$ are on the right-hand side of the bar. Also  
			$7$ must be on the right-hand side of the bar; namely, $7$ is not contained in the $p'$--string and there is a gap 
			between the $p'$-string and the $(p - p')$--part, so $7$ cannot be on the 
			left-hand side of the bar. This means that exactly one of the coordinates $6, 4, 1$ must be on the left-hand side of the bar. Therefore
			\begin{align*}
				\Lambda \in  \{(10, 9, 8; \ \mathbf{6}, 5, 3, 2, 0, -1, -5 \mid 
				18, 16, 15, 12, 11, 10, 9, 8, 7; \ 5, 4, 3, 2, 1, 0, -1), \\
				(10, 9, 8; \ 5, \mathbf{4}, 3, 2, 0, -1, -5 \mid 
				18, 16, 15, 12, 11, 10, 9, 8, 7, 6, 5; \, 3, 2, 1, 0, -1) \\
				(10, 9, 8; \ 5, 3, 2, \mathbf{1}, 0, -1, -5 \mid 
				18, 16, 15, 12, 11, 10, 9, 8, 7, 6, 5, 4, 3, 2; \ 0, -1), \}
			\end{align*}
			but none of these is unitary. Hence, there are no unitary parameters for which $p' = (10, 9, 8)$.
			
			Therefore, among the $\binom{10}{2} = 45$ $\frk$--dominant regular integral parameters conjugate to $\La^{\dom}$, only three are unitary. 
		}
	\end{ex}

	\section{The case of $\frso^{*}(2n)$}
	\label{sec:soz(2n)}
	In this case $\frg = \frso(2n, \bbC)$ and $\frk = \frgl(n, \bbC) $. We use standard coordinates and the standard positive root system. The positive compact roots are
	\[
	\eps_i - \eps_j, \quad 1 \leq i < j \leq n
	\]
	and the positive noncompact roots are
	\[
	\eps_i + \eps_j, \quad 1 \leq i < j \leq n.
	\]
	Therefore, the half sum of positive roots is 
	\[
	\rho = \left ( n-1, n-2,  \ldots,  0 \right). 
	\]
	As before, we will denote by $\La^{\dom}$ the $\frg$--dominant representative of the infinitesimal character $\La$, i. e., we have
	\[
	\La^{\dom}_1 \geq \La^{\dom}_2 \geq \ldots \geq \La^{\dom}_{n-1} \geq |\La^{\dom}_{n}|. 
	\]
	All parameters $\La$ corresponding to highest weight $(\frg, K)$-modules must be dominant regular for $\frk$. Therefore, we have\footnote{Note that this condition allows only for up to two nonzero coordinates to have the same absolute value.}
	\[
	\La_1 > \La_2 > \ldots >  \La_{n}.
	\]

	\begin{rem}\label{candidates}
		It will be convenient to consider as candidates not only conjugates of the parameter $\La^{\dom}$, but all 
		$\Lambda$ whose coordinates have the same absolute values as those of $\Lambda^{\mathrm{dom}}$. Among these, we will select the unitary ones, and for each such unitary $\Lambda$ we will determine whether it is conjugate to $\Lambda^{\mathrm{dom}}$ by counting the number of sign changes between $\La^{\dom}$ and $\La$; this number has to be even.\footnote{Of course sign change acts trivially on zero: $-0 = 0$ and therefore the number of visible sign changes can be odd.}
	\end{rem}
	
	The integrality of the parameter $\Lambda^{\mathrm{dom}}$ means that either all its coordinates are  integers, or all its coordinates are half-integers.
	
	\subsection{Integer coordinates}
	
	We will first deal with integer coordinates. In \cite{PPST3} it is shown that there are two series of unitary parameters -- the ``$q$-case'' and the ``$p$-case''. We first characterize these unitary parameters in a more convenient way depending on whether $0$ is among the coordinates or not and then determine which Weyl-group conjugates of a given $\La^{\dom}$ are unitary. We conclude this subsection with distinguishing the unitary points in the Hasse diagram.
	
	\cite[Theorem 5.6.]{PPST3} 
	implies that if $\la$ is in Case 1, that is, such that
	\[
	\la_1 > \la_2 = \ldots = \la_q, \quad q \in [2, \ldots, n]
	\]
	and $\la_q > \la_{q+1}$ in case $q < n$, then $L(\la)$ is unitary if and only if
	\[
	\la_1 + \la_2 \leq -2n + q + 2.
	\]
	We will refer to this case as the ``$q$-case".
	In terms of $\La = \la + \rho$, $\La$ is of the form
	\[
	\La=(\La_1, \underbrace{ \La_2, \La_2 - 1, \dots,\La_2 - (q-2)}_{q - 1},\La_{q+1},\dots,\La_n)
	\]
	for some $q\in[2,n]$, with $\La_1 >\La_2 + 1$ and $\La_q >\La_{q+1} + 1$ if $q<n$ . Then $\La$ is unitary if and only if
	\[
	\La_1 + \La_2 \leq q - 1.
	\]
	Now we have two cases.  
	
	\textbf{Case 1.1:} There does not exist a coordinate equal to $0$ in $\La^{\dom}$. Then $(\La_2, \ldots, \La_q)$ is a string that does not contain $0$. Therefore, either all $\La_2, \ldots, \La_q$ are positive, or all are negative. If $\La_2, \ldots, \La_q$ are positive, or equivalently $\La_q > 0$, then $\La_2 - (q - 2) \geq 1$ and 
	\[
	\La_1 + \La_2 \geq \La_2 + 2 + \La_2 \geq 2(q-1) + 2 > q - 1.
	\]
	Therefore, $\La$ is not unitary. 
	
	If $\La_2, \ldots, \La_q$ are all negative, or equivalently $\La_2 < 0$, then $\La$ is unitary if and only if $\La_1 + \La_2 \leq q - 1$, or equivalently $\La_1 \in \{ \La_2 + 2, \ldots, - \La_q + 1\}$.

	\textbf{Case 1.2:} There exists a coordinate equal to $0$ in $\La^{\dom}$. Let $x$ be the maximal number such that the coordinates of $\La^{\dom}$ include $0, 1, 1, \ldots, x, x$ ($x=0$ if $\La^{\dom}$ does not include coordinates 
	$1, 1$). Then $\La$ has to contain the string $x, x-1, \ldots, -x$. Now we have three cases.
	
	\textbf{Case 1.2.1:} If $\La_1 = 0$, then $x = 0$,  $\La_2 \leq -2$ and $\La_1 + \La_2 \leq -2 < q - 1$, so $\La$ is unitary.
	
	\textbf{Case 1.2.2:} If $q < n$ and the string $x, x-1, \ldots, -x$ occurs within the group of coordinates $\La_{q + 1}, \ldots, \La_n$, then $\La_q \geq \La_{q + 1} + 2 \geq x + 2 \geq 2$ and then $\La_2 = \La_q + (q-2) \geq q$. Therefore, 
	\[
	\La_1 + \La_2\geq 2\La_2+2\geq 2q+2>q-1,
	\]
	so $\La$ is not unitary.
	
	\textbf{Case 1.2.3:} If the string $x, x-1, \ldots, -x$ occurs inside the string $\La_2, \ldots, \La_q$, then due to the maximality of the number $x$, the string $x, x-1, \ldots, -x$ is either the rightmost part of $\La_2, \ldots, \La_q$, or the leftmost part of $\La_2, \ldots, \La_q$. In case the string $x, x-1, \ldots, -x$ is the rightmost part of $\La_2, \ldots, \La_q$, then $\La_q = -x$ and $2x + 1 \leq q-1$. Therefore we have
	\[
	\La_1 + \La_2 \geq \La_2 + 2 + \La_2 = 2(\La_q + (q-2)) + 2 = 2q-2 - 2x > q-1,
	\]
	so $\La$ is nonunitary.
	
	In case the string $x, x-1, \ldots, -x$ is the leftmost part of, and not all of, $\La_2, \ldots, \La_q$, then $\La_2 = x$ and $2x + 1 \leq q-2$ and the unitarity condition $\La_1 + \La_2 \leq q-1$ is equivalent to $\La_1 \in \{ \La_2 + 2, \ldots, 1 - \La_q\}$.
	
	We have proved the following proposition. 
	
	\begin{prop}
		\label{so*-case-1}  
		Suppose the coordinates of $\La^{\dom}$ are integers. 
		Let $\La$ be a $\frk$--regular conjugate of $\La^{\dom}$ of the form 
		\[
		\La=(\La_1, \underbrace{ \La_2, \La_2 - 1, \dots,\La_2 - (q-2)}_{q - 1},\La_{q+1},\dots,\La_n)
		\]
		for some $q\in[2,n]$, with $\La_1 >\La_2 + 1$ and $\La_q >\La_{q+1} + 1$ if $q<n$.
		\begin{itemize}
			\item [(1)] If $0$ is not a coordinate of $\La^{\dom}$, then $\La$ is unitary if and only if $\La_2 < 0$ and $\La_1 \in \{ \La_2 + 2, \ldots, 1 - \La_q\}$.
			\item [(2)] If $0$ is a coordinate of $\La^{\dom}$, let $x$ be the maximal number such that the coordinates of $\La^{\dom}$ include $0, 1, 1, \ldots, x, x$ ($x=0$ if $\La^{\dom}$ does not include coordinates 
			$1, 1$). Then $\La$ is unitary if and only if $\La_1 = x = 0$ or the string $x, x-1, \ldots, -x$ is the leftmost part of, and not all of, $\La_2, \ldots, \La_q$ and $\La_1 \in \{ \La_2 + 2, \ldots, 1 - \La_q\}$.
		\end{itemize}
	\end{prop}
	
	\bigskip
	
	On the other hand, \cite[Theorem 5.9. and Theorem 5.12.]{PPST3} 
	imply that if $\la$ is in Case 2, that is, such that 
	\[
	\la=(\underbrace{\la_1,\dots,\la_1}_p,\la_{p+1},\dots,\la_n)
	\]
	for some $p\in[2,n]$, with $\la_1>\la_{p+1}$ if $p<n$,
	then $L(\la)$ is unitary if and only if
	\[
	\lambda_1< -n+p-\left[\frac{p}{2}\right]+1=-n+\left[\frac{p+1}{2}\right]+1
	\]
	or $\la_1=-n+p-i+1$ for some $1\leq i\leq \left[\frac{p}{2}\right]$. We will refer to this case as the ``$p$-case". Since we assumed $\la_1 \in \bbZ$, $L(\la)$ is unitary if and only if $\la_1 \leq -n + p$.
	In terms of $\La = \la + \rho$, $\La$ is of the form
	\[
	\La=(\underbrace{\La_1, \La_1 - 1, \dots,\La_1 - (p-1)}_p,\La_{p+1},\dots,\La_n)
	\]
	for some $p\in[2,n]$, with $\La_1 - p >\La_{p+1}$ if $p<n$. Then $\La$ is unitary if and only if
	\[
	\La_1 \leq p - 1.
	\]
	We have two subcases.
	
	\textbf{Case 2.1:} $0$ is not a coordinate of $\Lambda^{\dom}$. In this case, for a $\frk$--regular $\Lambda \in W^1 \La^{\dom}$, either all coordinates of $\La$ are negative, or there is at least one positive coordinate of $\La$.
	
	If all coordinates of $\La$ are negative, then 
	\[
	\La_1 < 0 \leq p-1,
	\]
	so $\La$ is unitary.
	
	If there is at least one positive coordinate of $\La$, then since the coordinates of $\La$ decrease, $\La_1$ must be positive and since $\La_1, \ldots, \La_p$ is a string that does not contain $0$, then $\La_p > 0$. Therefore we have
	\[
	\La_1 = \La_p + (p - 1) \geq 1  + (p-1) = p.
	\]
	We conclude that $\La$ is nonunitary in this case.  
	
	\textbf{Case 2.2:} $0$ is a coordinate of $\Lambda^{\dom}$. 
	If the string $\La_1, \ldots, \La_p$ contains $0$, then $\La_p \leq 0$, i.e. $\La_1 - (p-1) \leq 0$. Therefore $\La_1 \leq p-1 $ so $\La$ is unitary. If the string $\La_1, \ldots, \La_p$ does not contain $0$, then $0 \in \{ \La_{p + 1}, \ldots, \La_{n}\}$. Therefore we have $\La_{p+1} \geq 0$ and $\La_1 - p > \La_{p+1} \geq 0$. Therefore, $\La_1 > p$ and $\La$ is not unitary.
	
	\begin{prop} 
		\label{so*-case-2}
		Suppose the coordinates of $\La^{\dom}$ are integers. 
		Let $\La$ be a $\frk$--regular conjugate of $\La^{\dom}$ of the form
		\[
		\La=(\underbrace{\La_1, \La_1 - 1, \dots,\La_1 - (p-1)}_p,\La_{p+1},\dots,\La_n),
		\]
		for some $p\in[2,n]$, with $\La_1 - p >\La_{p+1}$ if $p<n$.
		\begin{itemize}
			\item[(1)] If $0$ is not a coordinate of $\La^{\dom}$, then  $\La$ is unitary if and only if all coordinates of $\La$ are negative. 
			\item[(2)] If $0$ is a coordinate of $\La^{\dom}$, then $\La$ is unitary if and only if the string $\La_1, \ldots, \La_p$ contains $0$.
		\end{itemize}
	\end{prop}
	
	\begin{thm}\label{so*-zeroin}
		Suppose the coordinates of $\La^{\dom}$ are integers, and 
		that $0$ is a coordinate of $\Lambda^{\mathrm{dom}}$. 
		Let $x$ be the maximal integer such that 
		$0, 1, 1, \ldots, x, x$ occur among the coordinates of $\Lambda^{\mathrm{dom}}$ ($x=0$ if $\La^{\dom}$ does not include coordinates 
		$1, 1$). 
		Let $u \ge 0$ be the maximal integer such that 
		$x+1, \ldots, x+u$ are coordinates of $\Lambda^{\mathrm{dom}}$ 
		(with $u = 0$ when $x+1$ is not a coordinate of $\Lambda^{\mathrm{dom}}$).
		
		\begin{enumerate}
			
			\item If there is a repeated coordinate among the coordinates of $\Lambda^{\mathrm{dom}}$ that are greater than $x+u$, then $\Lambda^{\mathrm{dom}}$ has no unitary conjugates.
			
			\item Suppose that among the coordinates $x+2, \ldots, x+u$ the repeated ones are 
			$x+v_1, \ldots, x+v_r$ with $r \ge 2$, and that among the coordinates of 
			$\Lambda^{\mathrm{dom}}$ greater than $x+u$ there are no repeated coordinates. 
			Then the unitary parameters are of the form
			\begin{align*}
				\Lambda 
				= (& x+a, \ldots, x, x-1, \ldots, -x;\ 
				-x - v_1, \ldots, -x - v_r, \\
				& -x - a - 1, \ldots, -x - u,\ 
				-\Lambda^{\mathrm{dom}}_{n - (2x + u + r + 1)}, 
				\ldots, -\Lambda^{\mathrm{dom}}_{1}),
			\end{align*}
			where $v_r \le a \le u$. If $a = u$, then the string $ -x - a - 1, \ldots, -x - u$ is empty. 
			
			\item If among the coordinates $x+2, \ldots, x+u$ there is exactly one repeated 
			coordinate $x+v$ (with $v \ne 1$), and if among the coordinates of 
			$\Lambda^{\mathrm{dom}}$ greater than $x+u$ there are no repeated coordinates, 
			then the unitary parameters are of the form
			\[
			\Lambda 
			= (x+v;\ 
			x, x-1, \ldots, -x,\ -x-1, \ldots, -x-u;\ 
			-\Lambda^{\mathrm{dom}}_{n-(2x + u + 2)}, 
			\ldots, -\Lambda^{\mathrm{dom}}_{1})
			\]
			and
			\begin{align*}
				\Lambda 
				= (& x+a, \ldots, x, x-1, \ldots, -x;\ -x - v, \\
				& -x - a - 1, \ldots, -x - u,\ 
				-\Lambda^{\mathrm{dom}}_{n-(2x + u + 2)}, 
				\ldots, -\Lambda^{\mathrm{dom}}_{1}),
			\end{align*}
			where $v \le a \le u$. If $a = u$, then the string $ -x - a - 1, \ldots, -x - u$ is empty.
			
			\item If none of the coordinates greater than $x$ is repeated, 
			then the unitary parameters are of one of the following four forms:
			
			(a) If $u>1$, then for every $a\in[1,u-1]$ there is a unitary parameter in the $q$-case:
			\begin{align*}
				\La=&(x + a + 1;\ 
				x, \ldots, -x,\ -x-1, \ldots, -x-a;\\
				&\qquad \qquad -x - a - 2, \ldots, -x - u,\ 
				-\Lambda^{\mathrm{dom}}_{n-(2x + u + 1)}, 
				\ldots, -\Lambda^{\mathrm{dom}}_{1})
			\end{align*}
			
			(b) If $x = 0$ and $\La_{n-1}^{\mathrm{dom}}>1$ (so $u=0$), then there is one unitary parameter in the $q$-case:
			\[
			\La=(0;\ -\Lambda^{\mathrm{dom}}_{n-1}, \ldots, -\Lambda^{\mathrm{dom}}_{1}).
			\]
			
			(c) If $u\geq 1$, then for any $a\in[1,u]$ there is a unitary parameter in the $p$-case
			\begin{align*}
				\Lambda 
				= &(x+a, \ldots, x, x-1,\ldots,-x;\\ 
				& \qquad -x - a - 1, \ldots, -x - u,\ 
				-\Lambda^{\mathrm{dom}}_{n-(2x + u + 1)}, 
				\ldots, -\Lambda^{\mathrm{dom}}_{1}).
			\end{align*}
			
			(d) For any remaining cases of $\La^{\dom}$ such that $0$ is among its coordinates, there is a unitary parameter in the $p$-case
			\[
			\Lambda = ( x, x-1, \ldots, -x,\ -x-1, \ldots, -x-u;\ 
			-\Lambda^{\mathrm{dom}}_{n-(2x + u + 1)}, 
			\ldots, -\Lambda^{\mathrm{dom}}_{1}).
			\]
		\end{enumerate}
	\end{thm}
	
	\begin{proof}
		(1) If among the coordinates of $\Lambda^{\mathrm{dom}}$ that are greater than 
		$x+u$ there is a repeated coordinate, then all of them must appear in the 
		$\mathfrak{k}$-regular conjugate $\Lambda$ with both the $+$ and the $-$ sign. 
		
		If $\Lambda$ is a unitary conjugate belonging to the $q$--case, then 
		$\Lambda_1 \neq 0$, the string $x, x-1, \ldots, -x$ is the leftmost part of,
		but not the entire string $\Lambda_2, \ldots, \Lambda_q$, and 
		$\Lambda_1 \in \{\Lambda_2 + 2, \ldots, 1 - \Lambda_q\}$. 
		Therefore, $\Lambda_1$ must be a repeated coordinate greater than $x+u$, while 
		the $q$--string is of the form 
		\[
		x, x-1, \ldots, -x, \ldots, -x - a,
		\qquad 0 \le a \le u .
		\]
		But then
		\[
		\Lambda_1 + \Lambda_q \ge (x+u+2) + (-x - a) = u + 2 - a \ge 2,
		\]
		which implies $\La_1+\La_2\geq q$, and we conclude $\Lambda$ is not unitary. 
		
		If $\Lambda$ is a unitary conjugate belonging to the $p$--case, then the string 
		$x, x-1, \ldots, -x$ cannot be the leftmost part of the $p$--string, since there are repeated coordinates of $\La^{\dom}$ which are greater than $x$. 
		Hence, it must be the rightmost part of the $p$--string. 
		But then all repeated coordinates of $\Lambda^{\mathrm{dom}}$ that are greater 
		than $x+u$ must appear in the $p$--string with the $+$ sign. 
		This forces the $p$--string to contain the coordinate $x+u+1$, so $x+u+1$ is a coordinate of $\La^{\dom}$ which   
		contradicts maximality of $u$. Therefore, $\Lambda^{\mathrm{dom}}$ has no unitary conjugates in this case.
		\smallskip
		
		(2) Suppose that among the coordinates $x+2, \ldots, x+u$ the repeated ones are 
		$x+v_1, \ldots, x+v_r$ with $r \ge 2$, and that among the coordinates of 
		$\Lambda^{\mathrm{dom}}$ greater than $x+u$ there are no repeated coordinates. 
		Then all of the coordinates $x+v_1, \ldots, x+v_r$ must appear in the 
		$\mathfrak{k}$-regular conjugate $\Lambda$ with both the $+$ and the $-$ sign.
		
		If $\Lambda$ is a unitary conjugate belonging to the $q$--case, then 
		$\Lambda_1 \neq 0$, the string $x, x-1, \ldots, -x$ is the leftmost part of,
		but not the entire string $\Lambda_2, \ldots, \Lambda_q$, and 
		$\Lambda_1 \in \{\Lambda_2 + 2, \ldots, 1 - \Lambda_q\}$.
		
		The coordinates $x+v_1, \ldots, x+v_r$ must appear among the coordinates of $\Lambda$, 
		but there is room for at most one coordinate in the first position. 
		Therefore, in this case there are no unitary parameters belonging to the 
		$q$--case.
		
		If $\Lambda$ is a unitary conjugate belonging to the $p$--case, then the string 
		$x, x-1, \ldots, -x$ cannot be the leftmost part of the $p$--string, since there 
		are repeated coordinates of $\Lambda^{\mathrm{dom}}$ greater than $x$. 
		Hence, it must be the rightmost part of the $p$--string.
		
		Therefore, the unitary parameters are as in the statement (2) of the Theorem, i.e., of the form
		\begin{align*}
			\Lambda 
			= (& x+a, \ldots, x, x-1, \ldots, -x;\ 
			-x - v_1, \ldots, -x - v_r, \\
			&   -x - a - 1, \ldots, -x - u,\ 
			-\Lambda^{\mathrm{dom}}_{n - (2x + u + r + 1)}, 
			\ldots, -\Lambda^{\mathrm{dom}}_{1}),
		\end{align*}
		where $v_r \le a \le u$. 
		If $a = u$, then the string $-x - a - 1, \ldots, -x - u$ is empty.
		\smallskip
		
		(3) If among the coordinates $x+2, \ldots, x+u$ there is exactly one repeated 
		coordinate $x+v$ (with $v \ne 1$), and if among the coordinates of 
		$\Lambda^{\mathrm{dom}}$ greater than $x+u$ there are no repeated coordinates, then, using the same argument as in the previous case ($r = 1$), we conclude that the unitary parameters in the $p$--case are 
		\begin{align*}
			\Lambda 
			= (& x+a, \ldots, x, x-1, \ldots, -x;\ 
			-x - v, \\
			&   -x - a - 1, \ldots, -x - u,\ 
			-\Lambda^{\mathrm{dom}}_{n-(2x + u + 2)}, 
			\ldots, -\Lambda^{\mathrm{dom}}_{1}),
		\end{align*}
		where $v \le a \le u$. 
		If $a = u$, then the string $-x - a - 1, \ldots, -x - u$ is empty.
		
		If $\La$ is a unitary parameter in the $q$--case, then
		$\Lambda_1 \neq 0$, the string $x, x-1, \ldots, -x$ is the leftmost part of,
		but not the entire string $\Lambda_2, \ldots, \Lambda_q$, and 
		$\Lambda_1 \in \{\Lambda_2 + 2, \ldots, 1 - \Lambda_q\}$. Since $x + v$ is a repeated coordinate, it must appear in $\Lambda$ with both $+$ and the $-$ sign. Therefore, the parameter is of the form 
		\[
		\Lambda 
		= (x+v;\ 
		x, x-1, \ldots, -x,\ -x-1, \ldots, -x-u;\ 
		-\Lambda^{\mathrm{dom}}_{n - (2x + u + 2)}, 
		\ldots, -\Lambda^{\mathrm{dom}}_{1}).
		\]
		\smallskip
		
		(4) Suppose that none of the coordinates greater than $x$ is repeated.
		
		If $\Lambda$ is a unitary parameter in the $q$--case and $u > 0$, then the 
		$q$--string is of the form 
		\[
		x, x-1, \ldots, -x,\ -x-1, \ldots, -x-a,
		\qquad 1 \le a \le u.
		\]
		Since $\Lambda_1 + \Lambda_q \le 1$, we have 
		\[
		x + 2 \le \Lambda_1 \le x + a + 1.
		\]
		Because none of the coordinates $x+2, \ldots, x+a$ is a repeated coordinate of 
		$\Lambda^{\mathrm{dom}}$, we must have 
		\[
		\Lambda_1 \notin \{x+2, \ldots, x+a\}.
		\]
		Therefore $\Lambda_1 = x + a + 1$ for some $1 \le a < u$ 
		(the case $a = u$ is impossible, since $x + u + 1$ is not a coordinate of 
		$\Lambda^{\mathrm{dom}}$), and $\La$ is as in (a), i.e., 
		\begin{align*}
			\Lambda 
			= &(x + a + 1;\ 
			x, \ldots, -x,\ -x-1, \ldots, -x-a;\\
			&\qquad -x - a - 2, \ldots, -x - u,\ 
			-\Lambda^{\mathrm{dom}}_{n - (2x + u + 1)}, 
			\ldots, -\Lambda^{\mathrm{dom}}_{1}).
		\end{align*}
		If $a = u - 1$, then the string $-x - a - 2, \ldots, -x - u$ is empty.
		
		If $x = 0$ and $\Lambda_{n-1}^{\mathrm{dom}} > 1$, then there is only one 
		unitary parameter in the $q$--case, namely the one from (b):
		\[
		\Lambda = (0;\ -\Lambda_{n-1}^{\mathrm{dom}}, \ldots, -\Lambda_{1}^{\mathrm{dom}}).
		\]
		
		If $\Lambda$ is a unitary parameter belonging to the $p$--case, then the string 
		$x, x-1, \ldots, -x$ is either the rightmost or the leftmost part of the 
		$p$--string.
		
		If it is the rightmost part of the $p$--string, then $\Lambda$ is as in (c):
		\begin{align*}
			\Lambda 
			= (&x+a, \ldots, x, x-1, \ldots, -x;\\ 
			& -x - a - 1, \ldots, -x - u,\ 
			-\Lambda^{\mathrm{dom}}_{n-(2x + u + 1)}, 
			\ldots, -\Lambda^{\mathrm{dom}}_{1}),
		\end{align*}
		where $1 \le a \le u$ if $u \geq 1$. If $a = u$, then the string $-x - a - 1, \ldots, -x - u$ is empty. 
		
		If the string $x, x-1, \ldots, -x$ is the leftmost part of the $p$--string, 
		then $\Lambda$ is as in (d):
		\begin{align*}
			\Lambda 
			= (&x, x-1, \ldots, -x,\ -x-1, \ldots, -x-u;\ 
			-\Lambda^{\mathrm{dom}}_{n-(2x + u + 1)}, 
			\ldots, -\Lambda^{\mathrm{dom}}_{1}).
		\end{align*}
	\end{proof}
	
	\begin{thm}\label{so*-zeroout}
		Suppose the coordinates of $\La^{\dom}$ are integers, and  that $0$ is not a coordinate of $\La^{\dom}$.  
		\begin{enumerate}
			\item If among the coordinates 
			$\Lambda_1^{\mathrm{dom}}, \Lambda_2^{\mathrm{dom}}, \ldots, 
			\Lambda_{n-1}^{\mathrm{dom}}, |\Lambda_n^{\mathrm{dom}}|$ 
			at least two are repeated, then $\Lambda^{\mathrm{dom}}$ has no unitary conjugates.
			
			\item If among the coordinates 
			$\Lambda_1^{\mathrm{dom}}, \Lambda_2^{\mathrm{dom}}, \ldots, 
			\Lambda_{n-1}^{\mathrm{dom}}, |\Lambda_n^{\mathrm{dom}}|$ 
			exactly one is repeated, and if $\Lambda^{\mathrm{dom}}$ satisfies the following
			two conditions:
			\begin{itemize}
				\item[(i)] $(\Lambda_1^{\mathrm{dom}}, \ldots, \Lambda_{n-1}^{\mathrm{dom}}, 
				|\Lambda_n^{\mathrm{dom}}|)$ ends with a string in which the repeated coordinate appears twice;
				\item[(ii)] either $\Lambda_n^{\mathrm{dom}} > 0$ and $n$ is odd, 
				or $\Lambda_n^{\mathrm{dom}} < 0$ and $n$ is even,
			\end{itemize}
			then $\Lambda^{\mathrm{dom}}$ has exactly one unitary conjugate: the one that 
			starts with the repeated coordinate, followed by the remaining coordinates of
			\[
			\Lambda_1^{\mathrm{dom}},\ \Lambda_2^{\mathrm{dom}},\ \ldots,\ 
			\Lambda_{n-1}^{\mathrm{dom}},\ |\Lambda_n^{\mathrm{dom}}|
			\]
			taken with a negative sign and arranged in strictly decreasing order.
			
			If $\Lambda^{\mathrm{dom}}$ does not satisfy the above two conditions, then it has 
			no unitary conjugates.
			
			\item 
			
			If none of the coordinates 
			$\Lambda_1^{\mathrm{dom}}, \Lambda_2^{\mathrm{dom}}, \ldots, 
			\Lambda_{n-1}^{\mathrm{dom}}, |\Lambda_n^{\mathrm{dom}}|$ 
			is repeated, i.e., $\La^{\dom}$ is regular for $\frg$, then:
			\begin{itemize}
				\item[(i)] If $\Lambda_{n}^{\mathrm{dom}} > 0$ and $n$ is even, or 
				$\Lambda_{n}^{\mathrm{dom}} < 0$ and $n$ is odd, 
				then there is exactly one unitary conjugate of $\Lambda^{\mathrm{dom}}$:
				\[
				\Lambda = (-|\Lambda_n^{\mathrm{dom}}|,\ 
				-\Lambda_{n-1}^{\mathrm{dom}},\ \ldots,\ 
				-\Lambda_1^{\mathrm{dom}}).
				\]
				
				\item[(ii)] If $\Lambda_{n}^{\mathrm{dom}} > 0$ and $n$ is odd, or 
				$\Lambda_{n}^{\mathrm{dom}} < 0$ and $n$ is even, 
				let $S$ denote the longest string at the end of
				$(\Lambda_1^{\mathrm{dom}}, \ldots, \Lambda_{n-1}^{\mathrm{dom}},             |\Lambda_n^{\mathrm{dom}}|)$. Then the unitary conjugates of $\La^{\dom}$ are exactly those parameters that begin with one of the coordinates in $S$, 
				followed by the remaining coordinates of
				\[
				\Lambda_1^{\mathrm{dom}},\ \Lambda_2^{\mathrm{dom}},\ \ldots,\ 
				\Lambda_{n-1}^{\mathrm{dom}},\ |\Lambda_n^{\mathrm{dom}}|
				\]
				taken with a negative sign and arranged in strictly decreasing order.
			\end{itemize} 
		\end{enumerate}
	\end{thm}
	
	\begin{proof}
		(1) If among the coordinates 
		$\Lambda_1^{\mathrm{dom}}, \Lambda_2^{\mathrm{dom}}, \ldots, 
		\Lambda_{n-1}^{\mathrm{dom}}, |\Lambda_n^{\mathrm{dom}}|$ 
		at least two are repeated, then all of these repeated coordinates must appear in any 
		$\mathfrak{k}$--regular conjugate $\La$ of $\Lambda^{\dom}$ with both the plus and the minus sign.  
		Since $0$ is not a coordinate of $\Lambda^{\mathrm{dom}}$, every unitary parameter 
		may have at most one positive coordinate, because by Propositions \ref{so*-case-1} and \ref{so*-case-2}, the coordinates 
		$\Lambda_2, \ldots, \Lambda_n$ must be negative.  
		Therefore, $\Lambda^{\mathrm{dom}}$ has no unitary conjugates in this case.
		\smallskip
		
		(2) If among the coordinates 
		$\Lambda_1^{\mathrm{dom}}, \Lambda_2^{\mathrm{dom}}, \ldots, 
		\Lambda_{n-1}^{\mathrm{dom}}, |\Lambda_n^{\mathrm{dom}}|$ 
		exactly one is repeated, then there are no unitary conjugates in the $p$--case, 
		since by Proposition \ref{so*-case-2}(1) all unitary parameters in the $p$--case have all coordinates negative, 
		while the repeated coordinate $a$ must appear in any 
		$\mathfrak{k}$--regular conjugate of $\Lambda^{\dom}$ with both the plus and the minus sign.
		
		If $\Lambda$ is a unitary conjugate of $\La^{\dom}$ belonging to the $q$--case, then using Proposition \ref{so*-case-1}(1) we see that $\Lambda_1$ 
		must be equal to $a$, and since $\Lambda_1 + \Lambda_q \le 1$, we have 
		$\Lambda_q \le -a + 1$. 
		If $\Lambda_q = -a + 1$, then $-a$ does not appear in the $q$--string. 
		Since $-a$ must occur as a coordinate of $\Lambda$, it must lie among 
		$\Lambda_{q+1}, \ldots, \Lambda_n$; in fact, $\Lambda_{q+1} = -a$, and therefore 
		there is no gap between the $q$--string and the tail 
		$(\Lambda_{q+1}, \ldots, \Lambda_n)$, which is a contradiction.
		
		Therefore, $\Lambda_q \neq -a + 1$, and hence $\Lambda_q = -a - l$ for some 
		nonnegative integer $l$, and the $q$--string is of the form
		\[
		(-a + r, \ldots, -a, \ldots, -a - l),
		\]
		for some nonnegative integers $r$ and $l$, where $-a + r < 0$, since $\La_2$ is negative.
		Therefore, $\La_{q+1} \leq -a-l-2$. We conclude that the coordinates $a+l, \ldots, a-r$ occur, while the coordinates $a - r - 1$ and $a+l+1$ do not occur among the coordinates $\Lambda_1^{\mathrm{dom}}, \Lambda_2^{\mathrm{dom}}, \ldots, 
		\Lambda_{n-1}^{\mathrm{dom}}, |\Lambda_n^{\mathrm{dom}}|$. Furthermore, $a-r$ is the lowest among the coordinates $\Lambda_1^{\mathrm{dom}}, \Lambda_2^{\mathrm{dom}}, \ldots, 
		\Lambda_{n-1}^{\mathrm{dom}}, |\Lambda_n^{\mathrm{dom}}|$, i.e. $a-r = |\La_n^{\dom}|$.
		Therefore, to have any unitary conjugate, $\La^{\dom}$ must end with the coordinates $(a+l, \ldots, a+1, a, a, a-1, \ldots, a-r+1, a-r)$  or with the coordinates $(a+l, \ldots, a+1, a, a, a-1, \ldots, a-r+1, -a+r)$.
		
		Therefore, $\La$ must be of the form
		\[
		\La = (a; -a + r, \ldots, -a + 1, -a, -a-1, \ldots, -a-l;R),
		\]
		where $r$ and $l$ are the maximal integers such that  
		\[
		a + l,\ \ldots,\ a,\ a,\ a-1,\ \ldots,\ a - r
		\]
		all appear among 
		$\Lambda_1^{\mathrm{dom}}, \Lambda_2^{\mathrm{dom}}, \ldots, 
		\Lambda_{n-1}^{\mathrm{dom}}, |\Lambda_n^{\mathrm{dom}}|$,  $a-r = |\La_n^{\dom}|$
		and $R$ denotes the sequence consisting of the remaining coordinates of 
		$\Lambda^{\mathrm{dom}}$, written with a negative sign and arranged in 
		strictly decreasing order. Hence, 
		$R = (-\La^{\dom}_{n-(l+r+2)}, \ldots, -\La^{\dom}_{1})$
		and
		\begin{align}\label{canda}
			\La = (&\La_{n-r -1}^{\dom}; -|\La_n^{\dom}|, \ldots, -\La_{n-r+1}^{\dom}, -\La_{n-r}^{\dom}, -\La_{n-r - 2}^{\dom}, \ldots, -\La^{\dom}_{n-(l+r+1)};\\
			&\qquad \qquad\qquad\qquad\qquad -\La^{\dom}_{n-(l+r+2)}, \ldots, -\La^{\dom}_{1}), \notag
		\end{align}
		where $\La_{n-r - 1}^{\dom} = \La_{n-r}^{\dom} = a$, 
		\[
		\La^{\dom}_{n-(l+r+1)}, \ldots, \La_{n-r - 2}^{\dom}, \La_{n-r}^{\dom}, \La_{n-r+1}^{\dom},  \ldots, |\La_n^{\dom}|
		\]
		is a string and $\La^{\dom}_{n-(l+r+2)} - \La^{\dom}_{n-(l+r+1)} \geq 2.$

		Since the Weyl group $W$ consists of  permutations together with an even 
		number of sign changes, the parameter \eqref{canda} is a unitary conjugate of $\La^{\dom}$
		if and only if the number of sign changes from $\La^{\dom}$ to $\La$ is even. If $\La_{n}^{\dom} > 0$, then the number of sign changes is equal to $n-1$ and the parameter \eqref{canda} is  a unitary conjugate of $\La^{\dom}$ if and only if $n$ is odd.
		
		If $\La_{n}^{\dom} < 0$, then the number of sign changes is equal to $n-2$ and the  (unitary) parameter \eqref{canda} is a conjugate of $\La^{\dom}$ if and only if $n$ is even. 
		\smallskip
		
		(3) If none of the coordinates 
		$\Lambda_1^{\mathrm{dom}}, \Lambda_2^{\mathrm{dom}}, \ldots, 
		\Lambda_{n-1}^{\mathrm{dom}}, |\Lambda_n^{\mathrm{dom}}|$ is repeated, then one candidate for a unitary parameter is
		\begin{equation}\label{cand}
			\Lambda = (-|\Lambda_n^{\mathrm{dom}}|,\ -\Lambda_{n-1}^{\mathrm{dom}},\ \ldots,\ -\Lambda_1^{\mathrm{dom}}).
		\end{equation}
		The parameter \eqref{cand} is a unitary conjugate of $\La^{\dom}$
		if and only if the number of sign changes from $\La^{\dom}$ to $\La$ is even. 
		
		If $\Lambda_n^{\mathrm{dom}} > 0$, then the number of sign changes is $n$, and 
		\eqref{cand} is a unitary conjugate of $\La^{\dom}$  if and only if $n$ is even.
		
		If $\Lambda_n^{\mathrm{dom}} < 0$, then the number of sign changes is $n-1$, 
		and \eqref{cand} is a unitary conjugate of $\La^{\dom}$ if and only if $n$ is odd.
		The parameter \eqref{cand} is in the $p$--case if $\La_{n-1}^{\dom} = |\La_{n}^{\dom}| + 1$. If $\La_{n-1}^{\dom} \neq |\La_{n}^{\dom}| + 1$, then \eqref{cand} is in the $q$--case.
		
		Since by Proposition \ref{so*-case-2}(1) there are no further unitary candidates in the 
		$p$--case, we now derive a criterion for determining which of the remaining conjugates in the $q$--case are unitary. All other unitary candidates are of the form
		\[
		\La = (a; -b, -b-1, \ldots, -b-r; \La_{r + 3}, \ldots, \La_{n}),
		\]
		where $a, b, b+1, \ldots, b+r$ are among coordinates $\Lambda_1^{\mathrm{dom}}, \Lambda_2^{\mathrm{dom}}, \ldots, 
		\Lambda_{n-1}^{\mathrm{dom}}, |\Lambda_n^{\mathrm{dom}}|$, such that:
		
		\begin{itemize}
			\item $a>0$ (since we already treated the case with all coordinates of $\La$ negative);
			\item $\La_{r + 3} \leq -b-r-2$ (since we are assuming that $-b-r$ is the end of the $q$-string, i.e., $q=r+2$);
			\item $a\geq -b+2$, $a-b-r\leq 1$ (since by Proposition \ref{so*-case-2}(1), $\La_1\in [\La_2+2,1-\La_q])$.
		\end{itemize}
		
		It follows that $a \leq b+r+1$. Since none of the coordinates 
		$\Lambda_1^{\mathrm{dom}}, \Lambda_2^{\mathrm{dom}}, \ldots, 
		\Lambda_{n-1}^{\mathrm{dom}}, |\Lambda_n^{\mathrm{dom}}|$ is repeated, it follows that $a \notin \{ b, b+1, \ldots, b+r\}$. Therefore, either $0< a <b$ or $a = b + r + 1$. If $0<a<b$, then $a$ has the smallest absolute value among all of the coordinates of $\La^{\dom}$. Therefore, $a = |\La_n^{\dom}|$ and 
		\begin{equation}\label{cand2}
			\La = (|\La_n^{\dom}|; - \La_{n-1}^{\dom}, \ldots, - \La_{1}^{\dom}).
		\end{equation}
		
		If $\Lambda_n^{\mathrm{dom}} > 0$, then the number of sign changes from $\La^{\dom}$ to $\La$ is $n-1$, and 
		\eqref{cand2} is a unitary conjugate of $\La^{\dom}$ if and only if $n$ is odd.
		
		If $\Lambda_n^{\mathrm{dom}} < 0$, then the number of sign changes is $n$, 
		and \eqref{cand2} is a unitary conjugate of $\La^{\dom}$ if and only if $n$ is even.
		
		If $a = b + r + 1$, then $\Lambda$ is of the form
		\[
		\Lambda = 
		(b+r+1;\ -b,-b-1,\ldots,-b-r;\ 
		\Lambda_{r+3}, \ldots, \Lambda_n),
		\]
		where $\Lambda_{r+3} \le -b - r - 2$.
		Hence, $b$ has the smallest absolute value among all coordinates of $\Lambda$, 
		and therefore $b = |\Lambda_n^{\mathrm{dom}}|$.
		
		Let $k$ be the largest integer such that all coordinates 
		$b, b+1, \ldots, b+k$ occur among $\Lambda^{\dom}_{1}, \ldots, , \Lambda^{\dom}_{n-1}, |\Lambda^{\dom}_{n}|$.  
		Then
		\[
		b = |\Lambda_n^{\mathrm{dom}}|,\quad 
		b+1 = \Lambda_{n-1}^{\mathrm{dom}},\ \ldots,\ 
		b+k = \Lambda_{n-k}^{\mathrm{dom}},\quad
		b+k+1 < \Lambda_{n-k-1}^{\mathrm{dom}}.
		\]
		
		All unitary conjugates are therefore of the form
		\begin{align}\label{candr}
			\Lambda = (& b+r+1;\ -b,-b-1,\ldots,-b-r; \\
			& \qquad\qquad -b-r-2,\ldots,-b-k,
			-\Lambda_{n-k-1}^{\mathrm{dom}},\ldots,
			-\Lambda_1^{\mathrm{dom}}) \notag \\
			= & (\Lambda_{n - r- 1}^{\dom};\ -|\Lambda_{n}^{\dom}|, -\Lambda_{n - 1}^{\dom}, \ldots, - \Lambda_{n - r}^{\dom}, -\Lambda_{n - r - 2}^{\dom}, \ldots, -\Lambda_{1}^{\dom}),    \notag
		\end{align}
		where $0 \le r \le k-1$.
		
		If $\Lambda_n^{\mathrm{dom}} > 0$, then the number of sign changes from $\La^{\dom}$ to $\La$ is $n-1$, 
		and \eqref{candr} is a unitary conjugate of $\La^{\dom}$ if and only if $n$ is odd.  
		If $\Lambda_n^{\mathrm{dom}} < 0$, then the number of sign changes is $n-2$, 
		and \eqref{candr} is a unitary conjugate of $\La^{\dom}$ if and only if $n$ is even.
	\end{proof}
	
	\begin{ex}
		{\rm 
			Let $\La^{\dom}= ( 9, 4, 3, 3, 2, 1, 1, 0)$. Then each $\frk$--dominant regular conjugate of $\La^{\dom}$ must contain coordinates $3, 1, 0, -1, -3$, while the nonrepeated coordinates $9, 4$ and $2$ can appear with the positive or negative sign ($2^3 = 8$ possibilities). Therefore, $\frk$--dominant regular conjugates of $\La^{\dom}$ are 
			\begin{align*}
				& (9 \, | \,  4, 3, 2, 1, 0, -1 \, | \,  -3), (9 \, | \, 4, 3 \, | \, 1, 0, -1, -2, -3), \\
				& (9 \, | \, 3, 2, 1, 0, -1 \, | \, -3, -4), (9 \, | \, 3 \, | \, 1, 0, -1, -2, -3, -4), \\
				& (4, 3, 2, 1, 0, -1 \, | \, -3, -9), (4, 3 \, | \, 1, 0, -1, -2, -3,  -9), \\
				& (3, 2, 1, 0, -1 \, | \, -3, -4, -9), (3 \, | \, 1, 0, -1, -2, -3, -4 \, | \, -9) 
			\end{align*}
			We now apply Theorem \ref{so*-zeroin}. Since in this case $x = 1, u = 3$ and $v = 2$, unitary conjugates are $(3 \, | \, 1, 0, -1, -2, -3, -4 \, | \, -9)$, $(4, 3, 2, 1, 0, -1 \, | \,-3, -9)$ and $(3, 2, 1, 0, -1 \, | \, -3, -4, -9)$. All other conjugates are nonunitary.
		}
	\end{ex}
	
	\begin{ex}{\label{rho_so*}}
		{\rm
			Let $\La^{\dom}= (3, 2, 1, 0)$. Then $\frk$--dominant regular conjugates of $\La^{\dom}$ are 
			\begin{align*}
				& (3, 2, 1, 0), (3, 2 \, | \, 0, -1), (3 \, | \, 1, 0 \,| \,-2) \\  
				& (2, 1, 0 \,|\, -3), (3 \, | \, 0, -1, -2), (2 \,| \, 0, -1 \, | \, -3) \\
				& (1, 0 \,| \, -2, -3), (0, -1, -2, -3)
			\end{align*}
			We now apply Theorem \ref{so*-zeroin}.  Since in this case $x = 0, u = 3$, unitary conjugates are
			$(3 \, | \, 0, -1, -2), (2 \,| \, 0, -1 \, | \, -3)$, $(3, 2, 1, 0),  (2, 1, 0 \,|\, -3), (1, 0 \,| \, -2, -3)$ and $(0, -1, -2, -3)$.  All other conjugates are nonunitary.
		}
	\end{ex}
	
	Now we will generalize this example to an arbitrary $\rho$ and describe unitarity in terms of the Young diagrams.
	\begin{ex} 
		{\rm 
			Let $\Lambda^{\dom} = \rho=(n-1,n-2,\dots,0)$. Then all $\frk$-dominant $W$-conjugates of $\rho$ are of the form
			\[
			\La=w\rho=(i_1,\dots,i_k,0,-j_m,\dots,-j_1)
			\]
			for various choices of integers $k,m\in[0,n-1] $ and $i_1>\dots>i_k$, $j_1>\dots>j_m$, such that $k+m=n-1$ and
			\[
			\{i_1,\dots,i_k\}\cup\{j_1,\dots,j_m\}=\{1,\dots,n-1\}.
			\]
			Following \cite{H2}, we express these parameters in terms of Young diagrams
			with at most $n$ rows, with each row of length at most $n-1$, and with the additional condition that the diagrams can be built from hooks. Here by a hook (of size $k$) we mean the Young diagram with row lengths
			\[
			k,\underbrace{1,\dots,1}_k\qquad\text{for some}\quad k\in[1,n-1].
			\]
			By inverted hook we mean the diagram that is obtained from a hook by mirroring along ``antidiagonal''.
			A Young diagram $Y$ is built from hooks if it is either a hook, or the first row and column of $Y$ form a hook, and after removing that hook the resulting Young diagram can still be built from hooks.
			
			If the Young diagram $Y$ has row lengths $y_1\geq\dots \geq y_n$,  some of them possibly zero, we will write $Y=(y_1,\dots,y_n)$. The corresponding parameter $\La$ and highest weight $\la$ are
			\[
			\La=\rho-(y_n,\dots,y_1);\qquad \la=\La-\rho= -(y_n,\dots,y_1).
			\]
			\medskip
			
			{\rm Let $n=4$. The Hasse diagram of $\rho$ is
				
				\medskip
				
				{\footnotesize
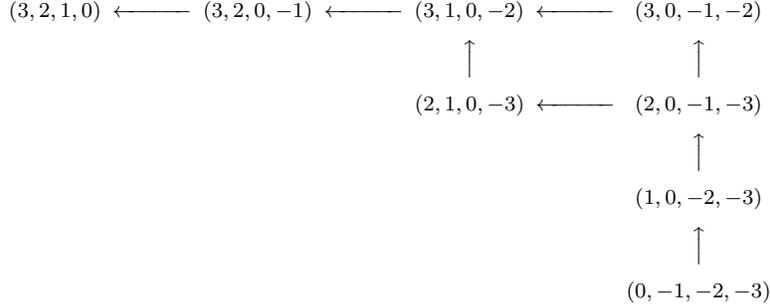
					\[
					\begin{CD}
						(3,2,1,0)@<<<(3,2,0,-1)@<<<(3,1,0,-2)@<<<(3,0,-1,-2) \\
						@. @. @AAA @AAA \\
						@. @. (2,1,0,-3)@<<<(2,0,-1,-3) \\
						@.@.@. @AAA \\
						@.@.@.(1,0,-2,-3) \\
						@.@.@.@AAA\\
						@.@.@.(0,-1,-2,-3)
					\end{CD}
					\]
				}
				\medskip
				
				\noindent with arrows pointing to larger elements in Bruhat order. The corresponding Young diagrams are
				
				\medskip
				
				{\tiny
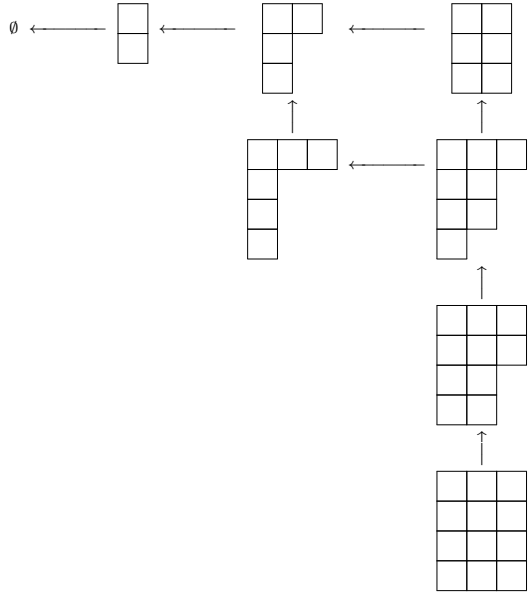
					\[
					\begin{CD}
						\emptyset @<<<\ydiagram{1,1}@<<<\ydiagram{2,1,1}@<<<\ydiagram{2,2,2} \\
						@. @. @AAA @AAA \\
						@.@.\ydiagram{3,1,1,1}@<<<\ydiagram{3,2,2,1} \\
						@.@.@. @AAA \\
						@.@.@.\ydiagram{3,3,2,2}\\
						@.@.@.@AAA\\
						@.@.@. \ydiagram{3,3,3,3}
					\end{CD}
					\]
				}
				\medskip
				
			}
			
			To determine the unitary points in the Hasse diagram of $\rho$, let us first assume that
			\[
			\La=(\La_1, \underbrace{ \La_2, \La_2 - 1, \dots,\La_2 - (q-2)}_{q - 1},\La_{q+1},\dots,\La_n)
			\]
			for some $q\in[2,n]$, with $\La_1 >\La_2 + 1$ and $\La_q >\La_{q+1} + 1$ if $q<n$.
			i.e. \[y_n<y_{n-1}=\dots=y_{n-q+1}.\]
			It follows from Theorem \ref{so*-zeroin}, where $x = 0$, $u = n-1$ that
			$\La$ is unitary if and only if $\La$ is of the form
			\begin{align}
				\label{q edge so*}
				\La & =(q-1\, | \, 0,\dots,-q+2\, | \, -q,\dots,-n+1),
			\end{align}
			where $3 \leq q \leq n$,
			with the dots standing for integers decreasing by one.
			We call the set of these points the $q$-edge of the Hasse diagram of $\rho$.
			
			Geometrically, each of the corresponding Young diagrams $Y$ is the full $n\times(n-1)$ box with an inverted hook of size $q-1$ taken out from the lower right corner. In Example \eqref{rho_so*}, there are two such diagrams, on top of the last column of the Hasse diagram. The other two diagrams in the last column of the Hasse diagram also follow the same pattern, with the inverted hook removed being of size $1$, respectively $0$; we do not include them here because the corresponding $\la$  does not belong to the $q$-case, but to the $p$-case that we will study next.
			\medskip
			
			Let us now assume that 
			\[
			\La=(\underbrace{\La_1, \La_1 - 1, \dots,\La_1 - (p-1)}_p,\La_{p+1},\dots,\La_n),
			\]
			for some $p\in[2,n]$, with $\La_1 - p >\La_{p+1}$ if $p<n$, i.e.
			\[
			y_n=\dots=y_{n-p+1}.
			\]
			Then $\La$ is unitary if and only if $\La$ is of the form
			\begin{align}
				\label{p edge so*}
				\La&=(p-1,\dots,1,0 \, | \, -p,-p-1,\dots,-n+1)
			\end{align}
			$p \in [2, n]$ or $\La=(0,-1,\dots,-n+1)$.
			We call the set of these points the $p$-edge of the Hasse diagram of $\rho$. Note also that for $p=n$, $Y$ is the empty diagram,  $\La=\rho$, and the corresponding module is the trivial module $\bbC$.
			
			In the remaining case, $y_n=n-1$, so $Y$ has all $n$ rows of length $n-1$, i.e., it is the $n\times(n-1)$ box. The corresponding parameter is 
			\eq
			\label{rhotil so*}
			\La=(0,-1,\dots,-n+1).
			\eeq
			This is of course $\tilde\rho$, the smallest element in the Hasse diagram. It is in the continuous part of its line, since \cite[Theorem 5.9 (3)]{PPST3} is clearly satisfied.
		}
	\end{ex}
	
	We have proved
	\begin{prop}
		\label{hasse rho so*}
		The unitary points of the Hasse diagram of $\rho$ belonging to the $q$-case are the points of the $q$-edge, given by \eqref{q edge so*} (these points are the last points of unitarity on their lines).
		The unitary points of the Hasse diagram of $\rho$ belonging to the $p$-case are:
		\begin{enumerate}
			\item $\tilde\rho$, given by \eqref{rhotil so*} (this point belongs to the continuous part of its line);
			\item the points of the $p$-edge, given by \eqref{p edge so*}
			(these points are the last points of unitarity on their lines). \qed
		\end{enumerate} 
	\end{prop}
	
	\subsection{Half--integer coordinates}
	
	Finally, assume that all coordinates of $\Lambda^{\mathrm{dom}}$ are half--integers.
	In the $q$--case, unitarity is equivalent to the condition
	\[
	\Lambda_1 + \Lambda_2 \leq q - 1.
	\]
	We distinguish two cases.
	
	\medskip
	\noindent
	\textbf{Case 3.1.}
	$\La^{\dom}$ does not contain two coordinates whose absolute value is equal to $\frac{1}{2}$. In this case $(\La_2, \ldots, \La_q)$ is a string that does not contain both $\frac{1}{2}$ and  $-\frac{1}{2}$. Therefore, either all coordinates $\La_2, \ldots, \La_q$ are positive, or all are negative. 
	
	If all coordinates $\La_2, \ldots, \La_q$ are positive, then $\La_q \geq \frac{1}{2}$, $\La_2 = \La_q + (q-2)$ and 
	\[
	\La_1 + \La_2 \geq \La_2 + 2 + \La_2 \geq 2(q-2 + \frac{1}{2} + 1) = 2q-1 > q-1.
	\]
	Therefore, $\La$ is not unitary. 
	
	If all coordinates $\La_2, \ldots, \La_q$ are negative, then $\La$ is unitary if and only if $\La_1 + \La_2 \leq q-1$, or equivalently if and only if $\La_1 \in \{ \La_2 + 2, \ldots, -\La_q + 1\}$.
	
	\medskip
	\noindent
	\textbf{Case 3.2.}
	$\La^{\dom}$ contains two coordinates whose absolute value is equal to $\frac{1}{2}$. Let $x$ be either $\half$, or the maximal coordinate such that $\La^{\dom}$ contains coordinates $x, x, \ldots, \frac{1}{2}, \frac{1}{2}$ or $x, x, \ldots, \frac{1}{2}, -\frac{1}{2}$, where $x > \frac{1}{2}$. If $x$ is in the tail $(\La_{q+1}, \ldots, \La_n)$, then 
	\[
	\La_2 = \La_q + (q-2) \geq \La_{q+1} + q \geq x + q \geq q + \frac{1}{2}
	\]
	and $\La_1 + \La_2 \geq 2 \La_2 + 2 \geq 2q + 3 > q-1$. Therefore, $\La$ is not unitary.
	
	If $x$ is in the $q$--string, then the $q$--string has to contain the coordinates $x, x-1, \ldots, \frac{1}{2}, - \frac{1}{2}, \ldots, -x$. Due to the maximality of the number $x$, the string
	\[
	x, x-1, \ldots, -x
	\]
	is either the rightmost part of $(\Lambda_2,\ldots,\Lambda_q)$ or the leftmost part of
	$(\Lambda_2,\ldots,\Lambda_q)$.
	
	If the string $x, x-1, \ldots, -x$ is the rightmost part of $(\Lambda_2,\ldots,\Lambda_q)$,
	then $\Lambda_q = -x$ and $2x+1 \le q-1$.
	Therefore, we have
	\[
	\Lambda_1 + \Lambda_2 \ge \Lambda_2 + 2 + \Lambda_2
	= 2(\Lambda_q + (q-2)) + 2
	= 2q - 2 - 2x\geq q > q - 1,
	\]
	so $\Lambda$ is nonunitary.
	
	If the string $x, x-1, \ldots, -x$ is the leftmost part of, but not equal to,
	$(\Lambda_2,\ldots,\Lambda_q)$, then $\Lambda_2 = x$ and $2x+1 \le q-2$.
	In this case, the unitarity condition $\Lambda_1 + \Lambda_2 \le q-1$ is equivalent to
	\[
	\Lambda_1 \in \{\Lambda_2 + 2, \ldots, 1 - \Lambda_q\}.
	\]
	We have proved
	
	\begin{prop}
		\label{so*-case-1-half-integral}    
		Assume that the coordinates of $\La^{\dom}$ are half-integers.
		Let $\La$ be a $\frk$--regular conjugate of $\La^{\dom}$ of the form 
		\[
		\La=(\La_1, \underbrace{ \La_2, \La_2 - 1, \dots,\La_2 - (q-2)}_{q - 1},\La_{q+1},\dots,\La_n)
		\]
		for some $q\in[2,n]$, with $\La_1 >\La_2 + 1$ and $\La_q >\La_{q+1} + 1$ if $q<n$.
		\begin{itemize}
			\item [(1)] If $\La^{\dom}$ does not contain two coordinates whose absolute value is equal to $\frac{1}{2}$, then $\La$ is unitary if and only if $\La_2 < 0$ and $\La_1 \in \{ \La_2 + 2, \ldots, 1 - \La_q\}$.
			\item [(2)] If $\La^{\dom}$ contains two coordinates whose absolute value is equal to $\frac{1}{2}$ and if $x$ is either $\half$, or the maximal number such that the coordinates of $\La^{\dom}$ include $x, x, \ldots, \frac{1}{2}, \frac{1}{2}$ or $x, x, \ldots, \frac{1}{2}, -\frac{1}{2}$, with $x>\half$, then $\La$ is unitary if and only if the string $x, x-1, \ldots, -x$ is the leftmost part of, and not all of, $\La_2, \ldots, \La_q$ and $\La_1 \in \{ \La_2 + 2, \ldots, 1 - \La_q\}$. \qed
		\end{itemize}
	\end{prop}
	
	If the coordinates of $\La^{\dom}$ are half-integers and if $\La$ is in the $p$--case, then $\La$ is unitary if and only if 
	\[
	\La_1 \leq \left [ \frac{p+1}{2} \right ].
	\]
	Since $\La_1 \in \frac{1}{2} + \mathbb{Z}$, the last condition is equivalent to 
	\[
	\La_1 \leq\frac{p}{2}. 
	\]
	Either all coordinates in the $p$--string are negative, or the $p$--string contains both  $\frac{1}{2}$ and $-\frac{1}{2}$. 
	(We ignore the possibility that all coordinates in the $p$--string are positive, because in that case $\La$ cannot be unitary.)
	
	If all coordinates in the $p$--string are negative, then, obviously, $\La$ is unitary.
	
	Let us assume that $\La^{\dom}$ contains two coordinates whose absolute value is equal to $\frac{1}{2}$ and, as before, let $x$ be either $\half$, or the maximal number such that the coordinates of $\La^{\dom}$ include $x, x, \ldots, \frac{1}{2}, \frac{1}{2}$ or $x, x, \ldots, \frac{1}{2}, -\frac{1}{2}$, such that $x>\half$. Any $\frk$--regular conjugate $\La$ has to contain the string $x, x-1, \ldots, -x$. If $x$ is in the tail $(\La_{p+1}, \ldots, \La_n)$, then $\La_p \geq x + 2$ and $\La_1 \geq x + p + 1 > p + 1$. Therefore, $\La$ is not unitary. 
	
	If the string $x, x-1, \ldots, -x$ is contained in the $p$--string, then, by the maximality of $x$,
	the string $x, x-1, \ldots, -x$ must be either the leftmost or the rightmost part of the
	$p$--string.
	
	If the string $x, x-1, \ldots, -x$ is the rightmost part of the $p$--string, then $\La_p = -x$ and $\La_1 = \La_p + (p-1) = -x + (p-1)$. Since the length of the string $x, x-1, \ldots, -x$ is $2x+1$, we have $2x+1 \leq p$. Unitarity condition is equivalent to $\La_1 \leq \frac{p}{2}$. Therefore, if $\La$ is unitary, we have
	\[
	-x + (p-1) = \La_1 \leq \frac{p}{2},
	\]
	which implies $2x + 2 \geq p$. Hence $p \in \{ 2x + 1, 2x + 2\}$ and the $p$--string is equal to the string
	\eq
	\label{p string 1}
	\left ( x, x-1, \ldots, \frac{1}{2}, -\frac{1}{2}, \ldots, -x \right),
	\eeq
	or to the string
	\eq
	\label{p string 2}
	\left ( x+1, x, x-1, \ldots, \frac{1}{2}, -\frac{1}{2}, \ldots, -x \right ).
	\eeq
	If the $p$-string is \eqref{p string 1}, i.e., $p=2x+1$, then 
	\[
	\La_1=x<x+\half=\frac{p}{2},
	\]
	and $\La$ is unitary. 
	
	If the $p$-string is \eqref{p string 2}, i.e., $p=2x+2$, then 
	\[
	\La_1=x+1=\frac{p}{2},
	\]
	and $\La$ is unitary.
	
	If the string $x, x-1, \ldots, -x$ is the leftmost part of, but not all of, the $p$--string, then $\La_1 = x$ and $2x+1 < p$. Therefore,
	\[
	\La_1 = x <  \frac{p-1}{2}<\frac{p}{2}.
	\]
	Therefore, $\La$ is unitary. 
	
	We have proved 
	
	\begin{prop} 
		\label{so*-case-2-half-integers}
		Assume that the coordinates of $\La^{\dom}$ are half-integers. Let $\Lambda$ be a $\frk$--regular conjugate of $\La^{\dom}$ of the form
		\[
		\La=(\underbrace{\La_1, \La_1 - 1, \dots,\La_1 - (p-1)}_p,\La_{p+1},\dots,\La_n),
		\]
		for some $p\in[2,n]$, with $\La_1 - p >\La_{p+1}$ if $p<n$.
		\begin{itemize}
			\item[(1)] If $\La^{\dom}$ does not contain two coordinates whose absolute value is equal to $\frac{1}{2}$, then  $\La$ is unitary if and only if all coordinates of $\La$ are negative. 
			\item[(2)] If $\La^{\dom}$ contains two coordinates whose absolute value is equal to $\frac{1}{2}$, let $x$ be either $\half$, or the maximal number such that the coordinates of $\La^{\dom}$ include $x, x, \ldots, \frac{1}{2}, \frac{1}{2}$ or $x, x, \ldots, \frac{1}{2}, -\frac{1}{2}$, such that $x>\half$.
			
			Then $\La$ is unitary if and only if the string  $x, x-1, \ldots, -x$  is the leftmost part of the string $\La_1, \ldots, \La_p$, or if
			\[
			(\La_1,\dots,\La_p) = \left ( x+1, x, x-1, \ldots, \frac{1}{2}, -\frac{1}{2}, \ldots, -x \right ). \qquad\qquad\qed
			\]
		\end{itemize}
	\end{prop}
	
	\begin{thm}\label{so*-halfinteger-onehalfoout}
		Assume that the coordinates of $\La^{\dom}$ are half-integers and that $\La^{\dom}$ does not contain two coordinates whose absolute value is equal to $\frac{1}{2}$.
		\begin{enumerate}
			\item If among the coordinates 
			$\Lambda_1^{\mathrm{dom}}, \Lambda_2^{\mathrm{dom}}, \ldots, 
			\Lambda_{n-1}^{\mathrm{dom}}, |\Lambda_n^{\mathrm{dom}}|$ 
			at least two are repeated, then $\Lambda^{\mathrm{dom}}$ has no unitary conjugates.
			
			\item If among the coordinates 
			$\Lambda_1^{\mathrm{dom}}, \Lambda_2^{\mathrm{dom}}, \ldots, 
			\Lambda_{n-1}^{\mathrm{dom}}, |\Lambda_n^{\mathrm{dom}}|$ 
			exactly one is repeated, and if $\Lambda^{\mathrm{dom}}$ satisfies the following
			two conditions:
			\begin{itemize}
				\item[(i)] $(\Lambda_1^{\mathrm{dom}}, \ldots, \Lambda_{n-1}^{\mathrm{dom}}, 
				|\Lambda_n^{\mathrm{dom}}|)$ ends with a string in which the repeated coordinate appears twice;
				\item[(ii)] either $\Lambda_n^{\mathrm{dom}} > 0$ and $n$ is odd, 
				or $\Lambda_n^{\mathrm{dom}} < 0$ and $n$ is even,
			\end{itemize}
			then $\Lambda^{\mathrm{dom}}$ has exactly one unitary conjugate: the one that 
			starts with the repeated coordinate, followed by the remaining coordinates of
			\[
			\Lambda_1^{\mathrm{dom}},\ \Lambda_2^{\mathrm{dom}},\ \ldots,\ 
			\Lambda_{n-1}^{\mathrm{dom}},\ |\Lambda_n^{\mathrm{dom}}|
			\]
			taken with negative signs and arranged in strictly decreasing order.
			
			If $\Lambda^{\mathrm{dom}}$ does not satisfy the above two conditions, then it has 
			no unitary conjugates.
			
			\item 
			
			If none of the coordinates 
			$\Lambda_1^{\mathrm{dom}}, \Lambda_2^{\mathrm{dom}}, \ldots, 
			\Lambda_{n-1}^{\mathrm{dom}}, |\Lambda_n^{\mathrm{dom}}|$ 
			is repeated, i.e., $\La^{\dom}$ is regular for $\frg$, then:
			\begin{itemize}
				\item[(i)] If $\Lambda_{n}^{\mathrm{dom}} > 0$ and $n$ is even, or 
				$\Lambda_{n}^{\mathrm{dom}} < 0$ and $n$ is odd, 
				then there is exactly one unitary conjugate of $\Lambda^{\mathrm{dom}}$:
				\[
				\Lambda = (-|\Lambda_n^{\mathrm{dom}}|,\ 
				-\Lambda_{n-1}^{\mathrm{dom}},\ \ldots,\ 
				-\Lambda_1^{\mathrm{dom}}).
				\]
				
				\item[(ii)] If $\Lambda_{n}^{\mathrm{dom}} > 0$ and $n$ is odd, or 
				$\Lambda_{n}^{\mathrm{dom}} < 0$ and $n$ is even, 
				let $S$ denote the longest string at the end of
				$(\Lambda_1^{\mathrm{dom}}, \ldots, \Lambda_{n-1}^{\mathrm{dom}},             |\Lambda_n^{\mathrm{dom}}|)$. Then the unitary conjugates of $\La^{\dom}$ are exactly those parameters that begin with one of the coordinates in $S$, 
				followed by the remaining coordinates of
				\[
				\Lambda_1^{\mathrm{dom}},\ \Lambda_2^{\mathrm{dom}},\ \ldots,\ 
				\Lambda_{n-1}^{\mathrm{dom}},\ |\Lambda_n^{\mathrm{dom}}|
				\]
				taken with negative signs and arranged in strictly decreasing order.
			\end{itemize} 
		\end{enumerate}
	\end{thm}
	\begin{proof}
		The proof follows the same arguments as the proof of Theorem \ref{so*-zeroout} and is therefore omitted.
	\end{proof}
	
	\begin{thm}\label{so*-halfinteger-onehalfin}
		Assume that the coordinates of $\La^{\dom}$ are half-integers and that $\La^{\dom}$ contains two coordinates whose absolute value is equal to $\frac{1}{2}$. Let $x$ be either $\half$, or the maximal number such that the coordinates of $\La^{\dom}$ include $x, x, \ldots, \frac{1}{2}, \frac{1}{2}$ or $x, x, \ldots, \frac{1}{2}, -\frac{1}{2}$, such that $x>\half$. Let $u \ge 0$ be the maximal integer such that 
		$x+1, \ldots, x+u$ are coordinates of $\Lambda^{\mathrm{dom}}$ 
		(with $u = 0$ when $x+1$ is not a coordinate of $\Lambda^{\mathrm{dom}}$).
		
		\begin{enumerate}
			
			\item If among the coordinates of $\Lambda^{\mathrm{dom}}$ that are greater than $x+u$ 
			there is a repeated coordinate, then $\Lambda^{\mathrm{dom}}$ has no unitary conjugates.
			
			\item Suppose that among the coordinates of 
			$\Lambda^{\mathrm{dom}}$ greater than $x+u$ there are no repeated coordinates, and that among the coordinates 
			$x+2, \ldots, x+u$ the repeated ones are 
			$x+v_1, \ldots, x+v_r$ with $r \ge 2$.  
			Then $\Lambda^{\mathrm{dom}}$ has no unitary conjugates. 
			
			\item Suppose that among the coordinates 
			of $\Lambda^{\mathrm{dom}}$ greater than $x+u$ there are no repeated coordinates, and that among the coordinates
			$x+2, \ldots, x+u$ there is exactly one repeated 
			coordinate $x+v$ (with $2\leq v \leq u$). Then the only possible unitary conjugate of $\La^{\dom}$ is
			\[
			\Lambda 
			= (x+v;\ 
			x, x-1, \ldots, -x,\ -x-1, \ldots, -x-u;\ 
			-\Lambda^{\mathrm{dom}}_{n-(2x + u + 2)}, 
			\ldots, -\Lambda^{\mathrm{dom}}_{1})
			\]
			Moreover, this parameter is unitary if and only if the number of sign changes 
			from $\La^{\dom}$ to $\La$ 
			is even.
			In particular, if the number of sign changes is odd, then $\La^{\dom}$ has no unitary conjugates.
			
			\item If none of the coordinates greater than $x$ is repeated, 
			the only possible unitary parameters are of one of the following forms.
			\begin{itemize}
				\item[(i)] If $u>1$, then for every $a\in[1,u-1]$ there is a 
				possible unitary conjugate of $\La^{\dom}$ in the $q$-case:
				\begin{align*}
					\La=&(x + a + 1;\ 
					x, \ldots, -x,\ -x-1, \ldots, -x-a;\\
					&\qquad \qquad -x - a - 2, \ldots, -x - u,\ 
					-\Lambda^{\mathrm{dom}}_{n-(2x + u + 1)}, 
					\ldots, -\Lambda^{\mathrm{dom}}_{1}).
				\end{align*}
				If $a = u-1$, then the string $-x- a - 2, \ldots, -x-u$ is an empty string.  
				
				\item[(ii)] If $u\geq 1$, then there is a possible unitary conjugate of $\La^{\dom}$ in the $p$-case of the form
				\begin{align*}
					\Lambda 
					= (& x+1, x, x-1, \ldots, -x; \ -x-2, \ldots, -x-u,
					-\Lambda^{\mathrm{dom}}_{n-(2x + u + 1)}, 
					\ldots, -\Lambda^{\mathrm{dom}}_{1}).
				\end{align*}
				If $u=1$, then the string $-x-2, \ldots, -x-u$ is an empty string.
				
				\item[(iii)] For any $\La^{\dom}$  there is a possible unitary conjugate of $\La^{\dom}$ in the $p$-case
				\[
				\Lambda = ( x, x-1, \ldots, -x,\ -x-1, \ldots, -x-u;\ 
				-\Lambda^{\mathrm{dom}}_{n-(2x + u + 1)}, 
				\ldots, -\Lambda^{\mathrm{dom}}_{1}).
				\]
				If $u=0$, then the string $-x-1, \ldots, -x-u$ is an empty string.
			\end{itemize}
			In each of the cases above, a parameter of the indicated form is a unitary conjugate of $\La^{\dom}$ if and only if the number of sign changes from $\La^{\dom}$ to $\La$ is even.
			
		\end{enumerate}
	\end{thm}
	\begin{proof}
		(1) The proof follows the same arguments as the proof of Theorem \ref{so*-zeroin} (1) and is
		therefore omitted.
		\smallskip
		
		(2)  As in Theorem~\ref{so*-zeroin} (2), we conclude that in this case there are no unitary parameters belonging to the \(q\)-case. Suppose that \(\Lambda\) is a unitary conjugate belonging to the \(p\)-case. By Proposition~\ref{so*-case-2-half-integers} (2), either the string
		\eq\label{string 1}
		x, x-1, \ldots, -x
		\eeq
		is the leftmost part of the \(p\)-string of $\La$, or
		\eq\label{string 2}
		(\Lambda_1,\dots,\Lambda_p)
		=
		\left(
		x+1, x, x-1, \ldots, \tfrac{1}{2}, -\tfrac{1}{2}, \ldots, -x
		\right).
		\eeq
		However, \eqref{string 1} cannot form the leftmost part of the \(p\)-string, since there are repeated coordinates of \(\Lambda^{\mathrm{dom}}\) greater than \(x\). So \eqref{string 2} must hold, but  this contradicts the assumption that there are at least two repeated coordinates greater than \(x\), and both of these must occur in \(\Lambda\) with both the plus and the minus sign.
		\smallskip
		
		(3) Assume that $\La$ is a unitary conjugate of $\La^{\dom}$. Since $x + v$ is a repeated coordinate, it must appear in $\La$ with both $+$ and the $-$ sign. 
		
		If $\La$ is in the $p$--case, then by Proposition~\ref{so*-case-2-half-integers} (2), either the string $x, x-1, \ldots, -x$ is the leftmost part of the $p$--string or \eqref{string 2} holds.
		Both cases lead to a contradiction, since the coordinate $x + v$ must appear in $\La$ and $v \neq 1$. Therefore, there are no unitary conjugates in the $p$--case. 
		
		If $\Lambda$ is a unitary parameter in the $q$--case, then by Proposition \ref{so*-case-1-half-integral}, the string $x, x-1, \ldots, -x$ is the leftmost part of, but not the entire string,
		$\Lambda_2, \ldots, \Lambda_q$, and
		\[
		\Lambda_1 \in \{\Lambda_2 + 2, \ldots, 1 - \Lambda_q\}.
		\]
		Since $x+v$ is a repeated coordinate, it must appear in $\Lambda$ with both 
		$+$ and the $-$ sign. Therefore, the parameter is of the form
		\[
		\Lambda = \bigl(
		x+v;\;
		x, x-1, \ldots, -x,\;
		-x-1, \ldots, -x-u;\;
		-\Lambda^{\mathrm{dom}}_{n-(2x+u+2)}, \ldots, -\Lambda^{\mathrm{dom}}_{1}
		\bigr).
		\]
		(4) Suppose that none of the coordinates greater than $x$ is repeated.
		
		If $\Lambda$ is a unitary parameter in the $q$--case and $u > 0$, then by Proposition \ref{so*-case-1-half-integral}, the 
		$q$--string $\La_2,\dots,\La_q$ is of the form 
		\[
		x, x-1, \ldots, -x,\ -x-1, \ldots, -x-a,
		\qquad \text{for some}\quad 1 \le a \le u.
		\]
		Furthermore, $\Lambda_1 + \Lambda_q \le 1$, so we have 
		\[
		x + 2 \le \Lambda_1 \le x + a + 1.
		\]
		Since none of the coordinates $x+2, \ldots, x+a$ is a repeated coordinate of 
		$\Lambda^{\mathrm{dom}}$, we must have 
		\[
		\Lambda_1 \notin \{x+2, \ldots, x+a\}.
		\]
		Therefore $\Lambda_1 = x + a + 1$ for some $1 \le a < u$ 
		(the case $a = u$ is impossible, since $x + u + 1$ is not a coordinate of 
		$\Lambda^{\mathrm{dom}}$), and $\La$ is as in (a), i.e., 
		\begin{align*}
			\Lambda 
			= &(x + a + 1;\ 
			x, \ldots, -x,\ -x-1, \ldots, -x-a;\\
			&\qquad -x - a - 2, \ldots, -x - u,\ 
			-\Lambda^{\mathrm{dom}}_{n - (2x + u + 1)}, 
			\ldots, -\Lambda^{\mathrm{dom}}_{1}).
		\end{align*}
		If $a = u - 1$, then the string $-x - a - 2, \ldots, -x - u$ is empty.
		
		If $u = 0$, then there are no unitary conjugates of $\La^{\dom}$ in the $q$--case.
		
		If $\Lambda$ is a unitary conjugate of $\La^{\dom}$ belonging to the $p$--case, then by Proposition \ref{so*-case-2-half-integers}, the string 
		$x, x-1, \ldots, -x$ is either the leftmost part of the 
		$p$--string $\La_1,\dots,\La_p$, or 
		\[
		(\Lambda_1,\dots,\Lambda_p)
		=
		\left(
		x+1, x, x-1, \ldots, \tfrac{1}{2}, -\tfrac{1}{2}, \ldots, -x
		\right).
		\]
		
		If the $p$--string is $x+1, x, x-1, \ldots, \tfrac{1}{2}, -\tfrac{1}{2}, \ldots, -x$, then  $\Lambda$ is as in (b):
		\begin{align*}
			\Lambda 
			= (& x+1, x, x-1, \ldots, -x;\ -x-2, \ldots, -x-u, \ 
			-\Lambda^{\mathrm{dom}}_{n-(2x + u + 1)}, 
			\ldots, -\Lambda^{\mathrm{dom}}_{1}).
		\end{align*}

		If the string $x, x-1, \ldots, -x$ is the leftmost part of the $p$--string, 
		then $\Lambda$ is as in (c):
		\begin{align*}
			\Lambda 
			= (&x, x-1, \ldots, -x,\ -x-1, \ldots, -x-u;\ 
			-\Lambda^{\mathrm{dom}}_{n-(2x + u + 1)}, 
			\ldots, -\Lambda^{\mathrm{dom}}_{1}).
		\end{align*}
	\end{proof}
	\begin{ex}{\rm 
			Let $\La^{\dom} = \left ( \frac{13}{2}, \frac{9}{2}, \frac{7}{2}, \frac{5}{2}, \frac{5}{2}, \frac{3}{2}, \frac{3}{2}, \frac{1}{2}, - \frac{1}{2} \right )$. Here $x = \frac{5}{2}$, $u = 2$ and from Theorem \ref{so*-halfinteger-onehalfin} it follows that the candidates for the unitary conjugates of $\La^{\dom}$ are
			\begin{align*}
				& \left (  \frac{9}{2}; \frac{5}{2},  \frac{3}{2}, \frac{1}{2}, - \frac{1}{2}, -\frac{3}{2}, -\frac{5}{2}, -\frac{7}{2}; -\frac{13}{2}\right ), \\
				& \left (  \frac{7}{2}, \frac{5}{2},  \frac{3}{2}, \frac{1}{2}, - \frac{1}{2}, -\frac{3}{2}, -\frac{5}{2}; -\frac{9}{2}, -\frac{13}{2}\right ), \\
				& \left (  \frac{5}{2},  \frac{3}{2}, \frac{1}{2}, - \frac{1}{2}, -\frac{3}{2}, -\frac{5}{2}, -\frac{7}{2}, -\frac{9}{2};  -\frac{13}{2}\right ).
			\end{align*}
			Since the first two have an even number of sign changes with respect to $\Lambda^{\mathrm{dom}}$, they  are unitary conjugates of $\La^{\dom}$. 
			The third one is not a unitary conjugate of $\La^{\dom}$, since the number of sign changes is odd.
		}
	\end{ex}

\end{document}